\documentclass[11pt]{amsart}
\usepackage{amsmath,amssymb}
\numberwithin{equation}{section}

\usepackage{xcolor}

\newtheorem{theorem}{Theorem}[section]
\newtheorem{thm}[theorem]{Theorem}
\newtheorem{lemma}[theorem]{Lemma}
\newtheorem{lem}[theorem]{Lemma}
\newtheorem{prop}[theorem]{Proposition}
\newtheorem{coro}[theorem]{Corollary}

\theoremstyle{definition}

\newtheorem{defi}[theorem]{Definition}

\newtheorem{remark}[theorem]{Remark}
 \newtheorem*{ackn}{Acknowledgements}
 
 \newtheorem*{thmA}{Theorem A} 
 \newtheorem*{thmB}{Theorem B} 
\newtheorem*{thmC}{Theorem C} 
\newtheorem*{thmD}{Theorem D} 
\newtheorem*{thmE}{Theorem E}

 \newcommand{\R}{\mathbb R}
 \newcommand{\Q}{\mathbb Q}
 \newcommand{\C}{\mathbb C}

 \newcommand{\e}{\varepsilon}

 \newcommand{\f}{\varphi}
 
 \newcommand{\p}{\psi}

 \newcommand \psh {{\rm PSH}}
 \newcommand \PSH {{\rm PSH}}

 \newcommand \Om {\Omega}

\newcommand{\Tr}{{\rm Tr}}

 \usepackage{hyperref}
\hypersetup{
    unicode=false,        
    pdftoolbar=true,      
    pdfmenubar=true,       
    pdffitwindow=false,     
    pdfstartview={FitH},    
    pdftitle={Uniform estimates for CMAE},    
    pdfauthor={Guedj, Lu},     
    colorlinks=true,       
   linkcolor=blue,          
    citecolor=blue,        
    filecolor=black,      
    urlcolor=blue}



\begin{document}

\title[Quasi-plurisubharmonic envelopes 3]{Quasi-plurisubharmonic envelopes 3: Solving Monge-Amp\`ere equations on hermitian manifolds}

\author{Vincent Guedj \& Chinh H. Lu}

\address{Institut de Math\'ematiques de Toulouse   \\ Universit\'e de Toulouse \\
118 route de Narbonne \\
31400 Toulouse, France\\}

\email{\href{mailto:vincent.guedj@math.univ-toulouse.fr}{vincent.guedj@math.univ-toulouse.fr}}
\urladdr{\href{https://www.math.univ-toulouse.fr/~guedj}{https://www.math.univ-toulouse.fr/~guedj/}}

\address{Universit\'e Paris-Saclay, CNRS, Laboratoire de Math\'ematiques d'Orsay, 91405, Orsay, France.}

\email{\href{mailto:hoang-chinh.lu@universite-paris-saclay.fr}{hoang-chinh.lu@universite-paris-saclay.fr}}
\urladdr{\href{https://www.imo.universite-paris-saclay.fr/~lu/}{https://www.imo.universite-paris-saclay.fr/~lu/}}
\date{\today}

 \begin{abstract}
 We develop a new   approach to $L^{\infty}$-a priori estimates for degenerate complex Monge-Amp\`ere equations on complex manifolds.
 It only relies on compactness and envelopes properties of quasi-plurisubharmonic functions.
  In a prequel \cite{GL21a} we have shown how this method allows one to 
 obtain new and efficient proofs of several fundamental
 results in K\"ahler geometry.
 In \cite{GL21b} we have studied the behavior of Monge-Amp\`ere volumes on hermitian manifolds.
We extend here the techniques of \cite{GL21a} to  the hermitian setting
and use the bounds established in \cite{GL21b}, producing
new relative a priori estimates, as well as several existence results for degenerate complex Monge-Amp\`ere
equations on compact hermitian manifolds.
 \end{abstract}

 \maketitle

\tableofcontents

\section*{Introduction}

Complex Monge-Amp\`ere equations have been one of the most powerful tools in K\"ahler geometry since 
Yau's solution to the Calabi conjecture \cite{Yau78}.
In recent years degenerate complex Monge-Amp\`ere equations have been intensively studied by many authors,
in relation to the Minimal Model Program
 (see \cite{GZbook,DonICM18,BBEGZ19,BBJ15} and the references therein).
 The main analytical input came here from pluripotential theory which allowed Ko{\l}odziej \cite{Kol98}
 to establish uniform a priori estimates 
in quite degenerate settings.

The study of complex Monge-Amp\`ere equations on compact hermitian  manifolds
was undertaken by Cherrier \cite{Cher87} and Hanani \cite{Han96}.
It has gained considerable interest in the last decade, after Tosatti and Weinkove
solved an appropriate version of Yau's theorem in \cite{TW10},
following important progress by Guan-Li \cite{GL10}
and geometric motivation for constructing special hermitian metrics 
(see e.g. \cite{FLY12}).
The smooth Gauduchon-Calabi-Yau conjecture has been solved   by
Sz\'ek\'elyhidi-Tosatti-Weinkove \cite{STW17}, while 
the pluripotential
theory has been partially extended by Dinew, Ko{\l}odziej, and Nguyen \cite{DK12,KN15,Din16, KN19},
allowing these authors to establish various uniform a priori estimates
in a non K\"ahler setting.

In \cite{GL21a} we have developed 
a new approach for establishing   uniform a priori estimates,
restricting to the context of K\"ahler manifolds for simplicity.
While the pluripotential approach consists in measuring the Monge-Amp\`ere capacity of
sublevel sets $(\f<-t)$, we directly measure the volume of the latter, avoiding
delicate integration by parts. Our approach thus applies in
  the hermitian   setting, providing several new results
that we now describe more precisely.

 \smallskip

We let $X$ denote a compact complex manifold of complex dimension $n$, equipped with a 
hermitian metric $\omega_X$. We fix $\omega$ a semi-positive $(1,1)$-form and set
$$
v_-(\omega):=\inf \left\{ \int_X (\omega+dd^c u)^n \; : \; u \in \PSH(X,\omega) \cap {L}^{\infty}(X)  \right\},
$$
where $d=\partial+\overline{\partial}$, $d^c=i (\overline{\partial}-\partial)$,
and $\PSH(X,\omega)$ is the set of $\omega$-plurisubharmonic functions: these are
functions $u$ which are locally   the sum of 
a smooth and a plurisubharmonic function,  and such that
$\omega+dd^c u \geq 0$ is a positive current.

When $\omega$ is closed, simple integration by parts reveal that $v_-(\omega)=\int_X \omega^n$
is positive as soon as the differential form $\omega$ is positive at some point. 
Bounding from below
$v_-(\omega)$ is a much more delicate issue in general which we discuss at length in \cite{GL21b}.
Our first main result is the following uniform a priori estimate when $v_-(\omega)$ is positive
(Theorem \ref{thm:uniformhermitian}):

\begin{thmA}
{\it 
 Let $\omega$ be semi-positive with $v_-(\omega)>0$.
   Let $\mu$ be a probability measure  such that
  $\PSH(X,\omega) \subset L^m(\mu)$ for some $m>n$.
Any solution $\f \in \PSH(X,\omega) \cap L^{\infty}(X)$  to 
 $(\omega+dd^c \f)^n=c \mu$, where $c>0$, satisfies
 $$
 {\rm Osc}_X (\f) \leq  T
 $$
 for some uniform constant $T$ which depends on  
an upper bound on $\frac{c}{v_-(\omega)}$ and
 $$
A_m({\mu}):= \sup \left\{\left( \int_X (-\p)^m d\mu\right)^{\frac{1}{m}} \; : \;  \p \in \PSH(X,\omega) \text{ with } \sup_X \p=0 \right\}.
 $$
 }
\end{thmA}

This result covers the case when $\mu=fdV_X$ is absolutely continuous with respect to Lebesgue measure,
with density $f$ belonging to $L^p$, $p>1$, or to an appropriate Orlicz class $L^w$ (for some convex weight $w$ with "fast growth" at infinity), 
thus partially extending the case of hermitian forms treated by Dinew-Ko{\l}odziej \cite{DK12} 
and Ko{\l}odziej-Nguyen \cite{KN15,KN20}.

We also provide a new and direct  alternative proof of this a priori estimate when
$\omega$ is hermitian, relying only on the local resolution
of the classical Dirichlet problem for the complex Monge-Amp\`ere equation,
and twisting the right hand side   with an exponential
(see   Theorem \ref{thm:globaluniformhermitian}).

\smallskip

As far as solutions to such equations are concerned,
we obtain the following:

\begin{thmB}
{\it 
Let $\omega$ be a semi-positive $(1,1)$ form which is either  big or such that $v_-(\omega)>0$.
Fix $0 \leq f \in L^p(dV_X)$, where $p>1$ and $\int_X fdV_X=1$.
Then  
\begin{itemize}
\item there exists a unique constant $c(\omega,f)>0$ 
and a bounded $\omega$-psh function $\f$ such that
$
(\omega+dd^c \f)^n=c(\omega,f) fdV_X;
$
\item for any $\lambda>0$ there exists a unique $\f_{\lambda} \in \PSH(X,\omega) \cap L^{\infty}(X)$
such that
$$
(\omega+dd^c \f_{\lambda})^n=e^{\lambda \f_{\lambda}}  fdV_X.
$$
\end{itemize}
}
\end{thmB}

By analogy with the K\"ahler setting, we say here that  $\omega$ is
{\it big} if   there exists  an $\omega$-psh function with analytic singularities $\rho$
such that $\omega+dd^c \rho \geq \delta \omega_X$  for some $\delta>0$.
%dominates a hermitian metric.
A celebrated result of Demailly-P\u{a}un \cite[Theorem 0.5]{DP04} 
      ensures, when $X$ is K\"ahler,
      that the existence of $\rho$ is a consequence of the condition $v_-(\omega)>0$.
    This result has been partially extended to the hermitian setting in \cite{GL21b}.

A slight refinement of our technique allows one to establish important stability estimates
(see Theorems  \ref{thm:stabilityhermitian} and  \ref{thm: stab sem pos}).
Similar results   have been obtained by quite different methods over the last
decade (see \cite{Blo11,DK12,Szek18,KN19,LPT20}).

\smallskip

There are several  geometric situations when 
one can not expect the Monge-Amp\`ere
potential $\f$ to be globally bounded. 
These corresponds to probability measures $\mu=fdV_X$
whose density $f$ belongs to an
Orlicz class $L^w$, for some 
convex weight $w$ with slow growth (see section \ref{sec:relative1}).
Our next main result provides
the following a priori estimate,
which extends to the hermitian setting
a result proved by  DiNezza-Lu \cite{DnL17}
in the context of quasi-projective varieties:

\begin{thmC} \label{thmC}
{\it 
Let $\omega$ be a semi-positive $(1,1)$ form which is big and such that $v_-(\omega)>0$.
Let  $\mu=fdV_X$ be a probability measure, where
$0 \leq f=g e^{-A\p} \in L^{w}$ with
$g \in L^p(dV_X)$ for some $p>1$, $A>0$ and  $\p \in \PSH(X, \omega)$.
Assume $\f$ is a bounded $\omega$-psh function
such that    $(\omega+dd^c \f)^n=c f dV_X$ and $\sup_X \f=0$. Then  
$$
 \alpha \p-\beta \leq \f \leq 0
$$
for any $0 < \alpha \leq 1$, where 
$\beta>0$ is a uniform constant that depends on $p$, the weight $w$, and upper bounds for $||g||_{L^p},\frac{A}{\alpha}, \frac{c}{v_-(\omega)}$, and the Luxembourg norm $\|f\|_w$.
}
\end{thmC}

The proof of DiNezza-Lu 
%\cite{DnL17}
 %in the K\"ahler case
  uses generalized Monge-Amp\`ere capacities, hence relies on Bedford-Taylor's comparison principle. 
% If $\omega$ is not closed,  
It has been  shown by Chiose \cite{Chi16} that this comparison principle 
  holds only under the restrictive condition $dd^c \omega = dd^c \omega^2 =0$.  
Our approach is completely different: we use a quasi-psh envelope construction to replace $f$ by $e^{-2A \f}$ 
(a similar idea has  been recently used in \cite{LN20,DDL5})
and then use Theorem A, whose proof also heavily relies 
on quasi-psh envelopes.
%To ensure uniformity of the constant $b$, we 
%show a priori bounds on the Lelong numbers using   an envelope technique and the assumption $v_-(\omega)>0$,
%and then invoke Skoda's integrability theorem \cite{Sko72}. 
% It seems likely that the latter assumption is not necessary but we leave it for a future project.  

\smallskip

% Approximating $\p$
%by quasi-psh functions with analytic singularities (following Demailly \cite{Dem92}),
%one can further show that $\f$ is locally bounded in a Zariski open subset of $X$.

We then move on to show the existence of   solutions
to such degenerate complex Monge-Amp\`ere equations.
In this context we prove the following :
%existence and regularity result:

\begin{thmD}
{\it 
Let $\omega$ be a semi-positive $(1,1)$ form which is big and such that $v_-(\omega)>0$.
%Let $\omega$ be a semi-positive   big $(1,1)$-form such that $v_-(\omega)>0$.
Fix $\rho \in \PSH(X,\omega)$ with analytic singularities along
a divisor $E$, such that $\omega+dd^c \rho$ dominates a hermitian form.
Let  $\mu=fdV_X$ be a probability measure, where
 $0 \leq f$ is smooth and positive in $X \setminus D$, and
   $f=e^{\p^+-\p^-}$ for some   quasi-plurisubharmonic functions $\p^{\pm}$.
%    with $\p^- \in PSH(X, a \omega)$ for some $a>0$.

Then there exist   $c>0$ 
 and $\f \in \PSH(X,\omega)$ such that   
\begin{itemize}
\item $\f$ is smooth in the open set $X \setminus (D \cup E)$;
\item $\alpha (\p^-+\rho) -\beta(\alpha) \leq \f $ in $X$ and $\sup_X \f=0$;
\item $(\omega+dd^c \f)^n=c fdV_X$ in $X \setminus (D \cup E)$;
\end{itemize} 
where $0<\alpha \leq 1$ is arbitrarily small, and $\beta(\alpha)$ is a uniform constant.
}
\end{thmD}

\smallskip

This result can be seen as a  generalization of the main result of \cite{TW10}.
It encompasses the case of smooth Monge-Amp\`ere equations on
mildly singular compact hermitian varieties, as well as more degenerate
settings, hermitian analogues of the main   results
of \cite[Theorems 1 and 3]{DnL17}.
It is obtained as a combination of  
Theorems \ref{thm:bdd} and   \ref{thm:higher}.
When moreover $f \in L^p(dV_X)$, $p>1$, then
$\f$ is globally bounded (one can take $\alpha=0$)
and it suffices to assume that $\omega$ is big (see Theorem \ref{thm:higher0}).

\smallskip

We finally apply our results to solve a singular version
of the hermitian Calabi-Yau theorem.
We work over a compact complex variety $V$ which has
{\it log terminal singularities} (see Section \ref{sec:geom} for a precise definition).
If the first Bott-Chern class $c_1^{BC}(V)$ vanishes, one says
that $V$ is a $\Q$-Calabi-Yau variety.
In that context we construct many 
Ricci-flat hermitian metrics:

\begin{thmE}
{\it 
Let $V$ be a 
%(non K\"ahler) 
$\Q$-Calabi-Yau variety
and $\omega_V$  a hermitian form.
There exists
a function $\f \in \PSH(V,\omega_V) \cap L^{\infty}(V)$ 
  such that
\begin{itemize}
\item $\f$ is  smooth in $V_{reg}$;
\item $\omega_V+dd^c \f$ is a hermitian form  and ${\rm Ric}(\omega_V+dd^c \f)=0$ in $V_{reg}$.
\end{itemize}
}
\end{thmE}

 We actually prove a singular hermitian analogue of Yau's celebrated solution to the Calabi conjecture,
 see Theorem \ref{thm:calabi}.
 We expect many further geometric implications of the present work, but
 leave this for a future project.

 \medskip

\noindent {\it Comparison with other works.}
%A key step towards solving degenerate complex Monge-Amp\`ere equations is an appropriate $L^{\infty}$-a priori estimate.
Yau's proof of his famous $L^{\infty}$-a priori estimate \cite{Yau78} goes through a delicate Moser iteration process.  
A PDE proof of the $L^{\infty}$-estimate has been 
%very 
recently provided by 
Guo-Phong-Tong \cite{GPT21} using an auxiliary 
Monge-Amp\`ere equation, inspired by the recent breakthrough
by Chen-Cheng on constant scalar curvature metrics \cite{CC21}.
An important generalization of Yau's estimate has been provided by Ko{\l}odziej \cite{Kol98} using pluripotential techniques.
This technique has been further generalized in \cite{EGZ09,EGZ08,DP10,BEGZ10}
in order to deal with less positive or collapsing  classes.
All these works require the underlying manifold to be K\"ahler.

B{\l}ocki has provided a different approach  \cite{Blo05b,Blo11} which is based on the Alexandroff-Bakelman-Pucci maximum principle
and a stability estimate due to Cheng-Yau ($L^2$-case) and Ko{\l}odziej ($L^p$-case).
B{\l}ocki's method works in the hermitian case 
and has been  generalized
to various settings by Sz\'ek\'elehydi, Tosatti and Weinkove   \cite{STW17, Szek18,TW18}. 
It requires the reference 
form to be strictly positive, 
but applies to a large family of equations
(see also \cite{TW15} for a slightly weaker $L^{\infty}$ estimate based on
 Moser iteration process).

The first steps of pluripotential theory have been developed in the hermitian setting by Dinew, Ko{\l}odziej and Nguyen
\cite{DK12,KN15,Din16, KN19}.
%\cite{DK14,Din16, KN16, KN19}.
 The presence of torsion requires
the reference form to be positive in order to control error terms in  delicate integration by parts.
%but provides quite general a priori estimates. 
%and efficient refined comparison principles under this positivity assumption.

Our approach consists in showing that the volume of the sublevel sets $(\f<-t)$ goes down to zero in finite time
by directly measuring their $\mu$-size. 
It relies on weak compactness  of normalized $\omega$-plurisubharmonic functions
and basic properties of quasi-psh envelopes, allowing us to deal
with semi-positive forms.

%We succeed in doing so by first composing $\f$ with a concave increasing weight $\chi$ and projecting $\chi \circ \f$ onto 
%the cone of $\omega$-plurisubharmonic function.
%The key lemma \ref{lem:GLkey} connects properties of the Monge-Amp\`ere measure of $P_{\omega}(\chi \circ \f)$ to that of $\f$.
%Roughly speaking one can choose highly increasing weights $\chi$ if 
%$MA_{\omega}(\f)$ has good integrability properties, and this forces $\f$ to be bounded.
%Besides properties of the quasi-plurisubharmonic envelope $P_{\omega}$, our approach only
%uses weak compactness  of normalized $\omega$-plurisubharmonic functions.

\begin{ackn} 
We thank D.Angella and {V.Tosatti}
for  useful discussions, as well as  T.D.T\^o and C.-M.Pan for a careful reading of a first draft.
%and  for his help concerning the ${\mathcal C}^2$ estimate. 
This work has benefited from State aid managed by the ANR under the "PIA" program bearing the reference ANR-11-LABX-0040
(research project HERMETIC).
 The authors are also partially supported by the ANR project PARAPLUI.
\end{ackn}

%\vfill
%\pagebreak[4]

\section{Preliminaries}

In the whole article we let $X$ denote a compact complex manifold of complex dimension $n \geq 1$,
equipped with a hermitian form $\omega_X$,
$dV_X$ a smooth probability measure, and
$\omega$ a smooth semi-positive $(1,1)$-form on $X$ such that $\int_X \omega^n>0$.

\subsection{Positivity properties and envelopes}

 \subsubsection{Monge-Amp\`ere measure}

  A function is quasi-plurisub\-harmonic 
  %(quasi-psh for short)
  if it is locally given as the sum of  a smooth and a psh function.

%Constant functions are $\omega$-psh functions since $\omega$ is semi-positive.
%A ${\mathcal C}^2$-smooth function $u$ has bounded Hessian, hence $\e u$ is
%$\omega$-psh if $0<\e$ is small enough and $\omega$ is  positive (i.e. hermitian).

\begin{defi}
 Quasi-plurisub\-harmonic functions
$\f:X \rightarrow \R \cup \{-\infty\}$ satisfying
$
\omega_{\f}:=\omega+dd^c \f \geq 0
$
in the weak sense of currents are called $\omega$-psh functions. 

We let $\PSH(X,\omega)$ denote the set of all $\omega$-plurisubharmonic ($\omega$-psh)
functions which are not identically $-\infty$.  
\end{defi}

We refer the reader to \cite{Dem12,GZbook,Din16} for basic
properties of $\omega$-psh functions.
Recall that:
\begin{itemize}
\item The set $\PSH(X,\omega)$ is a closed subset of $L^1(X)$, 
%when endowed with 
for the $L^1$-topology.
\item $\PSH(X,\omega) \subset L^r(X)$ for $r \geq 1$;
 the induced $L^r$-topologies are equivalent;
\item the subset $\PSH_A(X,\omega):=\{ u \in \PSH(X,\omega) \; : \; -A \leq \sup_X u \leq 0 \}$ is compact in
$L^r(X)$ for any $r \geq 1$ and all $A \geq 0$.
\end{itemize}

The complex Monge-Amp\`ere measure 
$
 (\omega+dd^c u)^n=\omega_u^n
$
 is well-defined for any
$\omega$-psh function $u$ which is {\it bounded}, as follows from Bedford-Taylor theory. 
%(see \cite{DK12}).

\smallskip

The mixed Monge-Amp\`ere measures 
$(\omega+dd^c u)^j \wedge (\omega+dd^c v)^{n-j}$ are also
well defined for any $0 \leq j \leq n$, and any bounded $\omega$-psh functions $u,v$.
%Among the important properties established by Bedford and Taylor, we shall often use the fact
%that the sets $(u<v)$ are open in the {\it plurifine topology}: although they are usually not
%open for the ambiant topology (the bounded $\omega$-psh function $v$ is
%not necessarily continuous), one has
%$$
%1_{\{u<v\}} (\omega+dd^c \max(u,v))^j \wedge (\omega+dd^c w)^{n-j} 
%= 1_{\{u<v \}} (\omega+dd^c v)^j \wedge (\omega+dd^c w)^{n-j}
%$$
%for any other bounded $\omega$-psh function $w$.
We recall the following classical inequality
(see \cite[Lemma 1.2]{GL21a}):

\begin{lem} \label{lem:Demkey}
Let $\f,\p$ be bounded $\omega$-psh functions such that $\f \leq \p$, then
$$
 1_{\{\p = \f \}} (\omega+dd^c \f)^j \wedge (\omega+dd^c \p)^{n-j} \leq 1_{\{\p = \f \}} (\omega+dd^c \p)^n,
$$
for all $1 \leq j \leq n$. 
\end{lem}

We shall also need the following generalization of the inequality of arithmetic and
 geometric means (see \cite[Lemma 1.9]{N16}):

\begin{lem} \label{lem:AGM}
Let $\f_1,\ldots,\f_n$ be bounded $\omega$-psh functions such that  
$(\omega+dd^c \f_i)^n\geq f_i dV_X$. Then
$
 (\omega+dd^c \f_1) \wedge \cdots \wedge (\omega+dd^c \f_n) \geq (\Pi_{i=1}^n f_i)^{\frac{1}{n}} dV_X.
$
\end{lem}

\subsubsection{Positivity assumptions}
%We  always assume in this article that $\omega$ is semi-positive and $\int_X \omega^n>0$. 
 On a few occasions we will need to assume slightly stronger positivity properties
 of the form $\omega$:

\begin{defi}
We say   {\it $\omega$ satisfies condition (B)} if there exists
 $B \geq 0$ such that
 $$
 -B \omega^2 \leq dd^c \omega \leq B \omega
 \; \; \text{ and } \; \;
 -B \omega^3 \leq d \omega \wedge d^c \omega \leq B \omega^3.
 $$
\end{defi}

Here are three different contexts where this condition is satisfied:
\begin{itemize}
\item any hermitian metric $\omega>0$ satisfies condition (B);
\item if $\pi:X \rightarrow Y$ is a desingularization of a singular compact complex variety $Y$
and  $\omega_Y$ is a hermitian metric, 
then $\omega=\pi^*\omega_Y$ satisfies condition (B);
\item if $\omega$ is semi-positive and closed, then it satisfies condition (B).
\end{itemize}
Combining these one obtains further settings where condition (B) is satisfied.

\begin{defi}
We say  that
\begin{itemize}
\item  {\it $\omega$ is non-collapsing} if 
 for any bounded $\omega$-psh function, the complex Monge-Amp\`ere 
 measure $(\omega+dd^c u)^n$ has positive mass: $\int_X (\omega+dd^c u)^n>0$;
 \item  {\it $\omega$ is uniformly non-collapsing} if $v_-(\omega)>0$, where
 $$
 v_-(\omega):=\inf \left\{ \int_X (\omega+dd^c u)^n \; : \; u \in \PSH(X,\omega) \cap L^{\infty}(X) \right\}.
 $$
\end{itemize}
 \end{defi}

These positivity notions are studied at length in \cite{GL21b}. It is shown there that
condition (B) implies non-collapsing, we further expect 
%that
 it implies uniform non-collapsing (at least in 
the case if $X$  belongs to the Fujiki class). 

\begin{defi}
We say   {\it $\omega$ is big} if 
there exists an $\omega$-psh function $\rho$ with analytic singularities
such that $\omega+dd^c \rho \geq \delta \omega_X$
dominates a hermitian form.
\end{defi}

If $V$ is a compact complex space endowed with a hermitian form $\omega_V$, and 
$\pi:X \rightarrow V$ is a resolution of singularities, then $\omega=\pi^* \omega_V$ is big.
This follows from classical arguments (see e.g. \cite[Proposition 3.2]{FT09}).

It is expected that $\omega$ is big if and only if $v_-(\omega)>0$.
This is a generalization of a conjecture  of Demailly-P\u{a}un \cite[Conjecture 0.8]{DP04},
which has been addressed in \cite{GL21b}: it is in particular shown 
in \cite[Theorem 4.6]{GL21b} that
if $v_+(\omega_X)<+\infty$ and $v_-(\omega)>0$ then $\omega$ is big, 
where
$$
v_+(\omega_X):=\sup \left\{ \int_X (\omega_X+dd^c u)^n \; :\;  u \in \PSH(X,\omega_X) \cap L^{\infty}(X) \right\}.
$$

\subsubsection{Envelopes} \label{sec:envelopes}

\begin{defi} \label{def:usual}
Given a Lebesgue measurable function $h:X \rightarrow \R$, we define the $\omega$-psh envelope of $h$ by
$$
P_{\omega}(h) := \left(\sup \{ u \in \psh (X,\omega) \; :\;  u \leq h  \, \,  \, \, X\}\right)^*,
$$
where the star means that we take the upper semi-continuous regularization.
\end{defi}

%We shall use the notation $P(h)=P_{\omega}(h)$ when no confusion can arise.
The following  has been established in \cite[Theorem 2.3]{GL21b}:

 \begin{thm} \label{thm:orthog}
If $h$ is bounded from below, quasi-l.s.c., and $P_{\omega}(h) <+\infty$, then
\begin{itemize}
\item $P_{\omega}(h)$ is a bounded $\omega$-plurisubharmonic function;
\item $P_{\omega}(h) \leq h$ in $X \setminus P$, where $P$ is pluripolar;
\item $(\omega+dd^c P_{\omega}(h))^n$ is concentrated on the contact  set $\{ P_{\omega}(h)=h\}$.
\end{itemize}
\end{thm}

A useful consequence is the following (see \cite[Lemma 2.5]{GL21b}):

\begin{lem} \label{lem:envmin}
Fix $\lambda\geq 0$ and let $u,v$ be bounded $\omega$-psh functions. 
Fix two smooth semi-positive $(1,1)$-forms $\omega_1,\omega_2$ such that $\omega_1 \geq \omega,\omega_2\geq \omega$. 
%and let $w=P(\min(u,v))$ denote the $\omega$-psh envelope of $\min(u,v)$. 

(i) If $(\omega_1+dd^c u)^n\leq e^{\lambda u}fdV_X$ and $(\omega_2+dd^c v)^n\leq e^{\lambda v}gdV_X$, then
$$
(\omega+dd^c P_{\omega}(\min(u,v)))^n \leq e^{\lambda P_{\omega}(\min(u,v))}\max(f,g) dV_X.
$$

(ii) If $(\omega+dd^c u)^n\geq e^{\lambda u}fdV_X$ and $(\omega+dd^c v)^n\geq e^{\lambda v}gdV_X$, then
$$
(\omega+dd^c \max (u,v)))^n \geq e^{\lambda \max(u,v))} \min(f,g) dV_X.
$$
\end{lem}

\begin{proof}
	The second statement follows from \cite[Lemma 2.5]{GL21b}. We prove the first one. Setting $\f:= P_{\omega}(\min(u,v))$, by Theorem \ref{thm:orthog} and Lemma \ref{lem:Demkey} we have 
	\begin{flalign*}
		(\omega+dd^c \f)^n &\leq {\bf 1}_{\{\f=u<v\}}(\omega+dd^c \f)^n + {\bf 1}_{\{\f=v\}} (\omega+dd^c \f)^n\\
		& \leq {\bf 1}_{\{\f=u<v\}}(\omega_1+dd^c \f)^n + {\bf 1}_{\{\f=v\}} (\omega_2+dd^c \f)^n\\
		&\leq {\bf 1}_{\{\f=u<v\}}(\omega_1+dd^c u)^n + {\bf 1}_{\{\f=v\}} (\omega_2+dd^c v)^n\\
		&\leq {\bf 1}_{\{\f=u<v\}} e^{\lambda \f} \max(f,g) dV_X + {\bf 1}_{\{\f=v\}} e^{\lambda \f} \max(f,g) dV_X\\
		&\leq e^{\lambda \f} \max(f,g) dV_X.
	\end{flalign*}
\end{proof}

The following is a key   tool to our new approach for uniform  estimates:
 
 \begin{lem} \label{lem:GLkey}
 Fix $\chi:\R^- \rightarrow \R^-$ a concave increasing function such that $\chi'(0) \geq 1$. 
  Let $\f,\phi$ be bounded $\omega$-psh functions with $\f \leq \phi$, and set $\p=\phi+\chi \circ (\f-\phi)$.
 Then
$$
   (\omega+dd^c P_{\omega}(\p))^n  \leq 1_{\{ P_{\omega}(\p)=\p\}} (\chi' \circ (\f-\phi))^n (\omega+dd^c \f)^n.
 $$
 \end{lem}
 
 The proof is identical to that of \cite[Lemma 1.6]{GL21a}, a consequence of
 Theorem \ref{thm:orthog} and Lemma \ref{lem:Demkey}.

 \subsection{Comparison and domination principles}
 
  \subsubsection{Comparison principle}

 The comparison principle plays a central role in K\"ahler pluripotential theory.
A "modified comparison principle" has been established by Ko{\l}odziej-Nguyen 
 \cite[Theorem 0.2]{KN15} when $\omega$ is hermitian.
  We extend the latter in  this section, assuming that 
  $\omega$ is merely big: we fix
   \begin{itemize}
   \item  an $\omega$-plurisubharmonic function $\rho$ with analytic singularities such that
     $\omega+dd^c \rho \geq  \delta \omega_X$ for some $\delta>0$; we set $\Omega:= \{\rho>-\infty\}$;
     \item   a constant $B_1>0$ 
     %depending on $\omega$, $\omega_X$, $n$
      such that for all $x \in \Omega$,
	\[
	-B_1  \omega_{\rho}^2 \leq dd^c \omega \leq B_1\omega_{\rho}^2,
	\text{ and }
	-B_1  \omega_{\rho}^3 \leq d \omega \wedge d^c \omega \leq B_1\omega_{\rho}^3. 
	\]
\end{itemize}    
    
 The existence of $B_1$ is clear since 
 $-B  \omega_{X}^2 \leq dd^c \omega \leq B\omega_{X}^2$
 for some $B>0$, and
 $-B  \omega_{X}^3 \leq d \omega \wedge d^c \omega \leq B\omega_{X}^3$.

	 \begin{theorem} \label{thm:pcpr}
%  \label{prop: CP for herm current}
		Assume $\omega$ is big and let		$\rho, B_1$ be as above.	
		Let $u$ be a bounded $\omega$-psh function, and  set $m:= \inf_X (u-\rho)$. Then for $s>0$ small enough we have 
	\[
	(1-4B_1s(n-1)^2)^n \int_{\{u <\rho +m+s\}} \omega_{\rho}^n   \leq \int_{\{u <\rho +m +s\}}  \omega_u^n.
	\]
	In particular, $\omega$ is non-collapsing. 
 \end{theorem}

%	The proof is  an extension of  that of \cite[Theorem 1.7]{GL21b}.

	\begin{proof}
			We set $\phi = \max(u,\rho+m+s)$ and $U:= \{u<\phi\}= \{u<\rho+m+s\}$. Observe that $U$ is relatively compact in the open set $\Omega$. For each $k\geq 0$ we set $T_k:= \omega_u^k \wedge \omega_{\phi}^{n-k}$ and $T_l=0$ for $l\leq 0$. Set $a=B_1 s(n-1)^2$. We prove by induction on $k=0,1,...,n-1$ that 
	\begin{equation}
		\label{eq: induction RCP}
		(1-4a) \int_U T_k \leq \int_U T_{k+1}.
	\end{equation}
	The conclusion follows since $\omega_{\phi}^n = \omega_{\rho}^n$ on the plurifine open set $U= \{u <\phi\}$. 
	
	We first prove \eqref{eq: induction RCP} for $k=0$. Since $u\leq \phi$
	  Lemma \ref{lem:Demkey} ensures that
	\[
	{\bf 1}_{\{u=\phi\}} \omega_{\phi}^n \geq {\bf 1}_{\{u=\phi\}} \omega_u \wedge \omega_{\phi}^{n-1}. 
	\]
	Noting that $X\setminus U = \{u=\phi\}$ we can write 
	\[
	\int_U (T_0-T_1) = \int_U (\omega_{\phi}^n - \omega_u\wedge \omega_{\phi}^{n-1}) \leq \int_X (\omega_{\phi}^n - \omega_u\wedge \omega_{\phi}^{n-1})=\int_X dd^c (\phi-u) \wedge \omega_{\phi}^{n-1}.  
	\]
Observe that
\begin{flalign*}
dd^c \omega_{\phi}^{n-1} & = (n-1) dd^c \omega \wedge \omega_{\phi}^{n-2} + n(n-1) d\omega \wedge d^c\omega \wedge \omega_{\phi}^{n-3} \\
& \leq 	B_1(n-1) \omega_{\rho}^2 \wedge  \omega_{\phi}^{n-2} + (n-1)(n-2)B_1 \omega_{\rho}^3 \wedge \omega_{\phi}^{n-3}. 
\end{flalign*}
Since $0\leq \phi-u \leq s$ and $U=\{u<\phi\}$, it follows from Stokes' theorem that 
\begin{flalign*}
	\int_X dd^c (\phi-u) \wedge \omega_{\phi}^{n-1} & = \int_X (\phi-u) dd^c \omega_{\phi}^{n-1}\\
%	&\leq \int_X (\phi-u)  (B_1(n-1) \omega_{\rho}^2 \wedge \omega_{\phi}^{n-2} + (n-1)(n-2)B_1  \omega_{\rho}^3 \wedge \omega_{\phi}^{n-3})\\
	&\leq s B_1(n-1)\int_U ( \omega_{\rho}^2 \wedge \omega_{\phi}^{n-2} + (n-2) \omega_{\rho}^3 \wedge \omega_{\phi}^{n-3})\\
	& \leq sB_1(n-1)^2 \int_U \omega_{\phi}^n,
\end{flalign*}
using that $\omega_{\rho}^k \wedge \omega_{\phi}^{n-k}=\omega_{\phi}^n$ on the plurifine
open set $U$. We thus get $\int_U (T_0-T_1) \leq sB_1(n-1)^2\int_U T_0$, proving \eqref{eq: induction RCP} for $k=0$. 

\smallskip

We now assume that \eqref{eq: induction RCP} holds for $j \leq k-1$, and we check that it still holds for $k$.  Observe that
\begin{eqnarray*}
\lefteqn{dd^c \left(\omega_u^k \wedge\omega_{\phi}^{n-[k+1]} \right)  } \\
& =&  k dd^c\omega \wedge\omega_u^{k-1} \wedge\omega_{\phi}^{n-[k+1]}
+(n-[k+1]) dd^c\omega \wedge \omega_u^k \wedge \omega_{\phi}^{n-[k+2]} \\
&+& 2 k(n-[k+1]) d\omega \wedge d^c \omega \wedge \omega_u^{k-1} \wedge \omega_{\phi}^{n-[k+2]} \\
&
+&k(k-1) d\omega \wedge d^c \omega \wedge \omega_u^{k-2} \wedge \omega_{\phi}^{n-[k+1]} \\
&+&  (n-[k+1])[n-(k+2)] d \omega \wedge d^c \omega\wedge \omega_u^{k} \wedge \omega_{\phi}^{n-[k+3]}.
\end{eqnarray*}
The same arguments as above therefore show that 
\begin{flalign*}
 \int_{U} (T_k-T_{k+1}) &\leq  \int_X (T_k-T_{k+1}) = \int_X (\phi-u) dd^c (\omega_{u}^{k}\wedge \omega_{\phi}^{n-[k+1]}) \\
&\leq B_1s\int_{U} \left(k(k-1)T_{k-2}+2k[n-k]T_{k-1}+(n-[k+1])^2 T_{k} \right) \\
&\leq a \left ( \frac{1}{(1-4a)^2} + \frac{1}{1-4a}+1 \right ) \int_{U}T_k \\
&\leq   4a   \int_{U} T_k,
\end{flalign*}
where   the third inequality uses the induction hypothesis,
while the fourth   follows from the upper bound $4a<1/8$.
 From this we obtain \eqref{eq: induction RCP} for $k$. 
 
 \smallskip
 
We finally prove that $\omega$ is non-collapsing.  If $u\in \PSH(X,\omega)$ is bounded and $\omega_u^n=0$ then the first statement of the proposition implies, since $\omega_{\rho}^n \geq \delta^n \omega_X^n$, 
that $\omega_{X}^n(u<\rho+m+s)=0$ for $s>0$ small enough.
 Since $u$ and $\rho$ are $\omega$-psh and $\omega_X>0$, this implies $u \geq \rho +m+s$, contradicting the definition of $m$.
	\end{proof}

  \subsubsection{Domination principle}
 
Several versions of the  domination principle have been established in 
\cite[Proposition 2.2]{LPT20}, \cite[Proposition 2.8]{GL21b}.
We shall   need the following generalization, valid  for mildly unbounded $\omega$-psh functions:

\begin{prop}\label{prop: domination principle unbounded}
	Fix $\rho$ an $\omega$-psh function with analytic singularities such that $\omega+dd^c \rho \geq \delta \omega_X$, with $\delta >0$. 
	Let $u,v$ be $\omega$-psh functions such that, for all $\varepsilon>0$, 
	$$\inf_X (\min(u,v)-\varepsilon \rho) >-\infty.$$ 
	 If $\omega_u^n \leq c \omega_v^n$ on $\{u<v\} \cap \{\rho>-\infty\}$ for some $c\in [0,1)$, then $u\geq v$. 
\end{prop}

The condition $\inf_X (\min(u,v)-\varepsilon \rho) >-\infty$ can be equivalently formulated as follows:
for any $\varepsilon>0$ there exists $C_{\e}>0$ such that
$u,v \geq \e \rho-C_{\e}$.
In particular  $u$ and $v$ are locally bounded in the Zariski open set $\Omega = \{\rho>-\infty\}$, 
hence the Monge-Amp\`ere measures $\omega_u^n$ and $\omega_v^n$ are  well-defined in $\Om$.

\begin{proof}
%	To simplify the notation we set $\Omega:= \{\rho>-\infty\}$ and we will work on this Zariski open set.
 We fix a constant $a>0$ so small that $\omega_u^n \leq c_1\omega_{\phi}^n$ on $\{u<v\}$, where $0<c_1<1$ and $\phi=(1-a) v + a\rho$. 
 By adding a constant we can assume that $v \geq \rho$ so that $\omega_u^n \leq c_1\omega_{\phi}^n$ on $\{u<\phi\}$. 
	
	We now fix   $b>1$ large and consider $\varphi := P_{\omega}(b u -(b-1)\phi)$. 
	It follows from	Theorem \ref{thm:orthog}  that $\varphi$ is bounded on $X$ and $\omega_{\varphi}^n$ is supported 
	on the contact set $\mathcal{C}:=\{\varphi= bu - (b-1)\phi\}$. Since $b^{-1}\varphi + (1-b^{-1}) \phi \leq u$ 
	with equality on the contact set, Lemma \ref{lem:Demkey} yields
	\[
	{\bf 1}_{\mathcal{C}} b^{-n} \omega_{\varphi}^n + {\bf 1}_{\mathcal{C}}(1-b^{-1})^n \omega_{\phi}^n \leq {\bf 1}_{\mathcal{C}} \omega_u^n. 
	\]
	Thus for $b>1$ large enough   $\omega_{\varphi}^n$ vanishes in $\mathcal{C}\cap \{u<\phi\}=\mathcal{C}\cap \{\varphi<\phi\}$, 
	hence on $\{\varphi<\max(\varphi,\phi)\}$ since $\omega_{\varphi}^n$ is supported on $\mathcal{C}$. 
	The domination principle ensures that $\varphi \geq \phi$. We infer   $u\geq \phi$ since 
	\[
	\varphi =P_{\omega}(bu-(b-1)\phi)) \leq bu-(b-1)\phi.
	\]  
	Thus $u\geq (1-a) v + a \rho$ and letting $a\to 0^+$ yields the conclusion. 
	\end{proof}

Here is a useful consequence of the domination principle. 

  \begin{coro} \label{cor:unique}
 Assume $\omega$ is non-collapsing and 
 let $u,v$ be bounded 
 $\omega$-psh functions.  If
 $$
  (\omega+dd^c u)^n \leq \tau  (\omega+dd^c v)^n 
 $$
 for some constant $\tau>0$, then $\tau \geq 1$.
 
 In particular  if  $(\omega+dd^c u)^n=c \mu$
 and $(\omega+dd^c v)^n =c' \mu$ for the same measure $\mu$,
 then $c=c'$.
 \end{coro}

\begin{proof}
 If $\tau<1$ the domination principle
 yields $u\geq v+C$ for any constant $C$, a contradiction. 
\end{proof}

The same result holds when $u,v$ are mildly singular:
%(compared to a hermitian potential $\rho$ with analytic singularities):

\begin{coro}\label{cor:unique2}
Assume $\omega$ is big and
fix $\rho$ an $\omega$-psh function with analytic singularities such that $\omega+dd^c \rho \geq \delta \omega_X$, with $\delta >0$. 
	Let $u,v$ be $\omega$-psh functions such that for all $\e>0$,
	%$u,v \geq \e \rho-C_{\e}$ for all $\varepsilon>0$ and for some $C_{\e}>0$.
	$$
	\inf_X (\min(u,v) - \varepsilon \rho) >-\infty. 
	$$
	\begin{enumerate}
	\item If $(\omega+dd^c u)^n=\tau (\omega+dd^c v)^n$ 
	%in $\{\rho >-\infty\}$ 
	then $\tau=1$.
	\item If $e^{-\lambda v} (\omega+dd^c v)^n \geq e^{-\lambda u} (\omega+dd^c u)^n$ for some $\lambda>0$, then $v \leq u$.
	\end{enumerate}
	\end{coro}
	
	\begin{proof}
	The proof of (1) is similar to that of Corollary \ref{cor:unique}, so we focus on (2)
 whose proof follows again from Proposition \ref{prop: domination principle unbounded}:
 fix $\delta>0$ and observe that in $\{u<v-\delta\} \cap \{\rho >-\infty\}$ we have
 $$
 (\omega+dd^c u)^n \leq e^{\lambda (u-v)}  (\omega+dd^c v)^n
 \leq e^{-\lambda \delta}  (\omega+dd^c (v-\delta))^n,
 $$
 with $c=e^{-\lambda \delta}<1$.
 Thus $v-\delta \leq u$ and the conclusion follows by letting $\delta \rightarrow 0$.
\end{proof}

\section{Uniform a priori estimates}

\subsection{Global $L^{\infty}$-bounds}

\subsubsection{Hermitian forms}
 
  When $\omega$ is a hermitian form and $\mu$ is a smooth volume form, it has been shown
     by Tosatti-Weinkove \cite{TW10} that there exists a unique
     $c>0$ and a unique smooth sup-normalized $\omega$-psh function $\f$ such that
     $$
     (\omega+dd^c \f)^n =c \mu.
     $$
     This landmark result relies on several previous attempts to generalize Yau's result to the hermitian setting, notably
     by Cherrier \cite{Cher87}, Hanani \cite{Han96} and Guan-Li \cite{GL10}.
     The key result that was missing and provided by \cite{TW10} is an a priori $L^{\infty}$-estimate for the solution.
          An alternative a priori estimate using pluripotential techniques has been provided by Dinew-Ko{\l}odziej in \cite{DK12},
     who treated the case when $\mu=f dV_X$ with $f \in L^p$, $p>1$.

     \begin{theorem} [Dinew-Ko{\l}odziej] \label{thm:globaluniformhermitian}
     Assume $\omega$ is a hermitian form, $p>1$ and $f \in L^p(dV_X)$ is such that
     $A^{-1} \leq \left(\int_X f^{\frac{1}{n}} dV_X\right)^n \leq \left(\int_X f^p dV_X \right)^{\frac{1}{p}}\leq A$ for some   $A>1$.
     If $c>0$ and $u \in \PSH(X,\omega) \cap L^{\infty}(X)$ are such that 
     $(\omega+dd^c u)^n=c fdV_X$, then
     $$
     c+c^{-1}+{\rm Osc}_X(u) \leq T,
     $$
     where the  constant $T$ only depends on $p,n,\omega$ and an upper bound for $A$.
     \end{theorem}
     
     The proof by Dinew-Ko{\l}odziej is a non trivial extension of the
     pluripotential approach developed by Ko{\l}odziej in the K\"ahler case \cite{Kol98},
     bypassing extra difficulties coming from the non closedness of $\omega$.
     We provide a direct proof of this result here,
     that only relies on  local resolutions of
     %complex 
     Monge-Amp\`ere equations.
     
     \begin{proof}
     
     {\it Step 1. Constructing a bounded subsolution.}
     We claim that there exist uniform constants $0<m=m(p,\omega)<M=M(p,\omega)$ such that for any
   $0 \leq g \in L^p$ with $\int_X g^p dV_X \leq 1$, we can
   find $v \in \PSH(X,\omega) \cap L^{\infty}(X)$ such that
   $$
   (\omega+dd^c v)^n \geq m g dV_X
   \; \; \text{ and } \; \;
   {\rm Osc}_X v \leq M.
   $$
     
    Consider indeed a finite double cover of $X$ by small "balls" $B_j,B_j'=\{\rho_j<0\}$, with $B_j \subset \subset B_j'$
    which are bounded in a local holomorphic chart. Here $\rho_j: X \rightarrow \R$ denotes a smooth function
  which is strictly plurisubharmonic in a neighborhood of $B_j'$.
    We solve $(dd^c v_j)^n=gdV_X$ in $B_j' $ with $-1$ boundary values.
    It follows from \cite{Kol98} that  the plurisubharmonic solution $v_j$ is uniformly bounded in $B_j'$.
   Considering $\max(v_j,\lambda_j \rho_j)$ we can choose  $\lambda_j >1$ and
    obtain a uniformly bounded function $w_j$ with the following properties:
    \begin{itemize}
    \item $w_j$ coincides with $v_j$ in $B_j$ where it satisfies $ (dd^c w_j)^n = g dV_X$;
    \item $w_j$ is plurisubharmonic in $B_j'$ and uniformly bounded;
    \item $w_j$ coincides with $\lambda_j \rho_j$ in $X \setminus B_j'$ and in a neighborhood of $\partial B_j'$.
    \end{itemize}
    As  $w_j$ is smooth where it is not plurisubharmonic, its curvature 
    $dd^c w_j$ is bounded below by $-\delta^{-1} \omega$ for some uniform $\delta>0$.
    Thus $\delta w_j$ is $\omega$-psh and
    $
    v:=\frac{\delta}{N} \sum_{j=1}^N w_j
    $
    is the bounded subsolution we are looking for, since in $B_j$ we obtain
    $$
    (\omega+dd^c v)^n \geq \frac{\delta^n}{N^n} (\omega+dd^c w_j)^n
    \geq \frac{\delta^n}{N^n} ( dd^c w_j)^n= \frac{\delta^n}{N^n} gdV_X.
    $$
     
     \medskip
     
 \noindent      {\it Step 2. Uniform a priori bounds.}
 We can normalize $u$ by $\sup_X u=0$.
 It follows from Skoda's uniform integrability  
(see \cite[Theorem 8.11]{GZbook}) that one can find $\e>0$, $p'=p'(\e,p) \in (1,p)$, 
and $C=C({\e},p)>0$ independent of $u$ such that
$g=e^{-\e u} f \in L^{p'}$ with
$$
||g||_{p'} \leq ||f||_p \cdot ||e^{-\e \frac{p}{p-p'}u}||_{\frac{p}{p-p'}} \leq C({\e},p) A.
$$

Let $v$ be the bounded subsolution provided by Step 1 for the density
$\frac{g}{||g||_{p'}}$, i.e.
$$
(\omega+dd^c v)^n \geq m \frac{g}{||g||_{p'}} dV_X \geq \frac{m'}{A} f dV_X=\frac{m'}{Ac} (\omega+dd^c u)^n,
$$
using that $u \leq 0$ hence $g \geq f$.   
It follows from Corollary \ref{cor:unique} that $c \geq m'/A$.
Note that the upper bound for $c$ follows easily from Lemma \ref{lem:AGM}.

\smallskip
 
 We finally observe that the uniform bound on $v$ also provides a uniform bound for $u$.
 Indeed since $v$ is bounded we obtain
 $$
 (\omega+dd^c v)^n \geq m'' e^{-\e u} fdV_X \geq e^{\e(v-u-C)} cfdV_X,
 $$
 hence Corollary \ref{cor:unique2} ensures that $u \geq v-C$.
     \end{proof}

  %   Going  one step further Ko{\l}odziej-Nguyen have shown in \cite{KN15} that    when $\omega$ is hermitian,
 %    there exists a continuous $\omega$-psh solution $\f$,    while uniqueness of $(c,\f)$ is obtained in \cite{KN19} assuming that $f \geq \delta_0>0$     stays bounded away from zero (see also \cite{KN20} for recent related results).
     
%     \smallskip
     
     %Using our new approach 

     \subsubsection{Semi-positive forms}
     
     We now extend this key $L^{\infty}$-estimate to the case
     when the form $\omega$ is not necessarily positive, assuming
    instead that $v_-(\omega)>0$:

   \begin{theorem}  \label{thm:uniformhermitian}
   Let $\omega$ be semi-positive with $v_-(\omega)>0$.
   Let $\mu$ be a probability measure  such that
  $\PSH(X,\omega) \subset L^m(\mu)$ for some $m>n$.
Any solution $\f \in \PSH(X,\omega) \cap L^{\infty}(X)$  to 
 $(\omega+dd^c \f)^n=c \mu$, where $c>0$, satisfies
 $$
 {\rm Osc}_X (\f) \leq  T
 %=M(\mu,c,v_-(\omega))
 $$
 for some uniform constant $T$ which depends on  
 upper bounds for $\frac{c}{v_-(\omega)}$ and  
 $$
A_m({\mu}):= \sup \left\{ \left(\int_X (-\p)^m d\mu \right)^{\frac{1}{m}} \; : \;  \p \in \PSH(X,\omega) \text{ with } \sup_X \p=0 \right\}.
 $$
 \end{theorem}

% We make the dependence $\mu \mapsto M(\mu,\cdot,\cdot)$ quite explicit as we shall need it in the sequel; it should
%also be useful in various K\"ahler settings.

Since any quasi-psh function belongs  to $L^r(dV_X)$ for all $r>1$,
  this theorem  applies   to measures $\mu=fdV_X$, where $f \in L^p$ with $p>1$,
  as follows from H\"older inequality.
  As in   \cite[Theorem 2.5.2]{Kol98} our technique also covers the case of more general densities.
We refer the reader to \cite[Section 2.2]{GL21a} for more details.

 \begin{proof}
 The proof is very similar to that of \cite[Theorem 2.1]{GL21a}.
 %We include it for the reader's convenience.
 Shifting by an additive constant we  can assume   $\sup_X \f=0$. Set 
 \[
 T_{\max}:= \sup \{t>0 \; : \; \mu(\f <-t) >0\}. 
 \]
 Our goal is to establish a precise bound on $T_{max}$.
    By definition,  $-T_{max} \leq \f$ almost everywhere with respect to $\mu$, 
    hence everywhere by the domination principle, 
    %(Proposition \ref{pro:domination}), 
  %  \marginpar{Condition (B) ?}
    providing the desired a priori bound    ${\rm Osc}_X (\f) \leq T_{max}$.

% It follows from compactness of sup-normalized $\omega$-psh functions
% and  Chebyshev inequality   that 
% there exists  $c_0$  independent of $\f$ such that
% for all $t >0$,
% \begin{equation} \label{eq:cheb1}
% {\mu }(\f<-t) \leq \frac{||\f||_{L^1(\mu)}}{t}
% \leq \frac{2||\f||_{L^1(\mu)}}{1+t}
%\leq  \frac{c_0}{1+t},
% \end{equation}
% using that $ {\mu }(\f<-t)=0$ for $0 \leq t \leq 1$ and $t+1 \leq 2$ for $t \geq 1$.
% 

 We let $\chi:\R^- \rightarrow \R^-$ denote a {\it concave} increasing function
 such that $\chi(0)=0$ and $\chi'(0) = 1$. 
 We set $\p=\chi \circ \f$, $u=P(\p)$ and observe that
%\begin{eqnarray*}
$$
\omega+dd^c \p = \chi' \circ \f \omega_\f+[1-\chi' \circ \f] \omega+\chi'' \circ \f d\f \wedge d^c \f 
\leq  \chi' \circ \f \omega_\f.
$$
%\end{eqnarray*}
It follows from Lemma \ref{lem:GLkey} that
 $$
\frac{1}{c} (\omega+dd^c u)^n  
 \leq 1_{\{u=\p\}}  (\chi' \circ \f)^n \mu.
 $$

 \smallskip
 
\noindent  {\it Controlling the energy of $u$}.
We fix $\e>0$ so small that 
$
n<n+3\e =m.
$
 The concavity of $\chi$ and the normalization $\chi(0)=0$ yields $|\chi(t)| \leq |t| \chi'(t)$. 
 Since $u=\chi \circ \f$ on the support of $(\omega+dd^c u)^n$, H\"older's inequality yields
 \begin{eqnarray*}
 \int_X (-u)^{\e} \frac{(\omega+dd^c u)^n }{c}  
 &\leq &  \int_X (-\chi \circ \f)^{\e} (\chi' \circ \f)^n d\mu   
 \leq   \int_X (- \f)^{\e} (\chi' \circ \f)^{n+\e} d\mu   \\
 & \leq &  \left( \int_X (-\f)^{n+2\e} d\mu \right)^{\frac{\e}{n+2\e}} 
 \left( \int_X ( \chi' \circ \f)^{n+2\e} d\mu \right)^{\frac{n+\e}{n+2\e}} \\
 & \leq & A_m(\mu)^{\e}  \left( \int_X ( \chi' \circ \f)^{n+2\e} d\mu \right)^{\frac{n+\e}{n+2\e}} 
 \end{eqnarray*}
  using that $\f$ belongs to the set of $\omega$-psh functions $v$  normalized by $\sup_X v=0$ 
 which is compact in $L^{n+2\e}(\mu)$.
 
  \smallskip
 
\noindent  {\it Controlling the norms $||u||_{L^m}$}.
We are going to choose below the weight $\chi$ in such a way that
$\int_X ( \chi' \circ \f)^{n+2\e} d\mu =B \leq 2$ is a finite constant under control.
This provides a uniform lower bound on $\sup_X u$ as we now explain:
indeed
\begin{eqnarray*}
0 \leq  (-\sup_X u)^{\e} \, \frac{v_-(\omega)}{c} 
&\leq & (-\sup_X u)^{\e} \int_X  \frac{(\omega+dd^c u)^n }{c} \\
&\leq & \int_X (-u)^{\e} \frac{(\omega+dd^c u)^n }{c}  \\
&\leq & 2 A_m(\mu)^{\e}
\end{eqnarray*}
 yields
 \begin{equation*} %\label{eq:supvol}
  -\left( \frac{2c }{v_-({\omega})} \right)^{1/\e} A_m(\mu) \leq \sup_X u \leq 0.
 \end{equation*}

We infer that  $u$ belongs to a compact set of $\omega$-psh functions, hence 
$$
 ||u||_{L^m(\mu)} 
 %\leq A_m(\mu)+\left( \frac{2c }{v_-({\omega})} \right)^{1/\e} A_m(\mu) 
 \leq \left[1+\left( \frac{2c }{v_-({\omega})} \right)^{1/\e} \right] A_m({\mu})=:\tilde{A}^m.
 $$
%since $r=n+3\e<m$.
From $u \leq \chi \circ \f \leq 0$ we get $||\chi \circ \f||_{L^m} \leq ||u||_{L^m} \leq \tilde{A}^m$. 
Using Chebyshev inequality we thus obtain
\begin{equation} \label{eq:clef}
 {\mu }(\f<-t) \leq \frac{\tilde{A}}{|\chi|^m(-t)}.
\end{equation}

   \smallskip
 
\noindent  {\it Choice of $\chi$}.
Recall that if $g: \R^+\rightarrow \R^+$  is   increasing with $g(0)=1$, then
$$
\int_X g \circ (-\f) d\mu = \mu(X) + \int_0^{T_{\max}} g'(t) {\mu }(\f<-t) dt.
$$
%Fix $0<T_0<T_{\max}$. 
Setting $g(t)=[\chi' (-t)]^{n+2\e}$  we define $\chi$ by imposing $\chi(0)=0$, $\chi'(0)=1$, and 
$$
g'(t)=
%\begin{cases}
	\frac{1}{(1+t)^2 {\mu }(\f<-t)} \; \text{ if}\;  t <T_{max}.
	%\\	0 \; \; \; \; \text{ if}\;  t \geq  T_{max}
%\end{cases}.
$$
This choice guarantees that $\chi$ is concave increasing with $\chi' \geq 1$ on $\f(X)$, and
$$
\int_X ( \chi' \circ \f)^{n+2\e} d\mu \leq  \mu(X) + \int_0^{+\infty} \frac{dt}{(1+t)^2}\leq 2.
$$
 A slight technical point is that one should first work on a compact subinterval $[0,T']$ with $0<T'<T_{max}$
 and then let $T'$ tend to $T_{max}$. We refer the reader to the proof of 
 \cite[Theorem 2.1]{GL21a} for more details on this twist.

\smallskip

 \noindent  {\it Conclusion}.  We set $h(t)=-\chi(-t)$ and work with the positive counterpart of $\chi$. 
 Note that $h(0)=0$ and $h'(t)=[g(t)]^{\frac{1}{n+2\e}}$ is positive increasing, hence
$h$ is convex.
% increasing (so $\chi$ is concave increasing and negative). 
Observe also that $g(t) \geq g(0)=1$ hence $h'(t) =[g(t)]^{\frac{1}{n+2\e}} \geq 1$ yields
\begin{equation} \label{eq:minh(1)}
h(1)=\int_0^1 h'(s) ds \geq 1.
\end{equation}

Together with \eqref{eq:clef} our choice of $\chi$ yields, for all $t\in [0,T_{max})$, 
$$
\frac{1}{(1+t)^2g'(t)}={\mu }(\f<-t) \leq \frac{\tilde{A}}{h^m(t)}.
$$
This reads
$$
h^m(t) \leq \tilde{A} (1+t)^2 g'(t)=(n+2\e) \tilde{A} (1+t)^2 h''(t)  (h')^{n+2\e-1}(t).
$$
Multiplying  by $h'$,  integrating between $0$ and $t$, 
we infer that for all $t \in [0,T_{max})$,
\begin{eqnarray*}
\frac{h^{m+1}(t)}{m+1}
&\leq&   (n+2\e)  \tilde{A}   \int_0^t (1+s)^2 h''(s)  (h')^{n+2\e}(s)\\
&\leq & \frac{(n+2\e) \tilde{A} (t+1)^2 }{n+2\e+1} \left ((h')^{n+2\e+1}(t) - 1 \right)  \\
&\leq&    \tilde{A}    (1+t)^2   (h')^{n+2\e +1}(t)  .
\end{eqnarray*}
Recall that we have set   $m=n+3 \e$ so that 
$$
\alpha:=m+1= (n+2\e +1)+\e> \beta:= n+2\e +1>2.
$$
The previous inequality then reads
$
 {(1+t)^{-\frac{2}{\beta}}} \leq C  {h'(t)}{h(t)^{-\frac{\alpha}{\beta}}},
$
for some uniform constant $C$ depending on $n,m,\tilde{A}$.
Since $\alpha>\beta>2$ and $h(1) \geq 1$ (by \eqref{eq:minh(1)}), 
integrating the above inequality between $1$ and $T_{max}$ we obtain 
$
T_{max} \leq C', 
$
for some uniform constant $C'$ depending on $C,\alpha,\beta$. 
%Since $T_0$ was chosen arbitrarily in $]0,T_{\max}[$ the result follows. 
 \end{proof}

\subsubsection{Stability estimate}

Adapting similarly the proof   of \cite[Theorem 2.4]{GL21a}, we  also obtain:
 
   \begin{theorem}  \label{thm:stabilityhermitian}
   Let $\omega,\mu$ be as in Theorem \ref{thm:uniformhermitian}. 
Let $\f \in \PSH(X,\omega) \cap L^{\infty}(X)$  be such that $\sup_X \varphi=0$ and
 $(\omega+dd^c \f)^n=c \mu$ for some $c>0$.  Then 
 $$
\sup_X (\phi-\varphi)_+ \leq  T \left (\int_X (\phi-\varphi)_+d\mu \right)^{\gamma},
 $$
for any $\phi \in \PSH(X,\omega)\cap L^{\infty}(X)$, where $\gamma=\gamma(n,m)>0$
and $T$
 is a uniform constant which depends on 
upper bounds for $\frac{c}{v_-(\omega)}$,
  $||\phi ||_{L^{\infty}}$, and
 $$
A_m({\mu}):= \sup \left\{ \left(\int_X (-\p)^m d\mu\right)^{\frac{1}{m}} \; : \; \p \in \PSH(X,\omega) \text{ with } \sup_X \p=0 \right\}.
 $$
 \end{theorem}

 If $\phi \in \PSH(X,\omega) \cap L^{\infty}(X)$ also satisfies $(\omega+dd^c \phi)^n=c' \mu'$ for some $c'>0$
 and $\mu'$ having similar properties as that of $\mu$, the above result yields an
 $L^1-L^{\infty}$-stability estimate,
 $$
||\phi-\f||_{L^{\infty}(X)} \leq  M ||\phi-\varphi||_{L^1(\mu+\mu')}^{\gamma}.
 $$
 Thus  if a sequence of such solutions $\f_j$  converges in $L^1(\mu)$, this should allow one to conclude that
 it actually uniformly converges. We shall use this strong information on several occasions in the sequel.

\subsection{Relative  $L^{\infty}$-bounds} \label{sec:relative1}

%   Let  $\omega$ be a semi-positive and big form.
 We consider in this section the degenerate complex Monge-Amp\`ere equation 
 \begin{equation} \label{eq:geom}
  (\omega+dd^c \f)^n=c   fdV_X,
 \end{equation} 
where  
%the density 
$0 \leq f \in L^1(X)$.
%does not belong to any    "good Orlicz class". 
It follows from abstract measure theoretic arguments that  $f$ belongs to an Orlicz class
$L^w(dV_X)$ 
for some convex increasing weight $w:\R^+ \rightarrow \R^+$
such that $\frac{w(t)}{t} \rightarrow +\infty$ as $t \rightarrow +\infty$.
The Luxemburg norm
$$
||f||_w:=\inf \left\{ r>0 \; : \;  \int_X w \left( \frac{f(x)}{r} \right) dx <1 \right\}
$$
is finite if and only if
$\int_X w \circ f \, dV_X <+\infty$.

%This Orlicz class is "good" 
If the weight $w$ grows fast enough at infinity
(e.g. $w(t)=t^p$ with $p>1$, or $w(t)=t (\log t)^m $ with $m>n$),
then   one can expect the solution  to be uniformly bounded
(see \cite[Theorem 2.5.2]{Kol98}).
We assume here  this is {\it not} the case.

\smallskip

%Thus one cannot expect  anylonger that $\f$ is globally bounded on $X$.
 %Since   $\mu=fdV_X$ is non pluripolar, it follows from \cite{GZ07,Din09}
% that there exists a unique $\omega$-psh solution $\f$ in the finite energy class ${\mathcal E}(X,\omega)$,
%normalized by $\sup_X \f=0$.
%However 

When   $f \leq e^{-\p}$ for some quasi-psh function $\p$,
and $X$ is K\"ahler, it has been shown by DiNezza-Lu 
 \cite[Theorem 2]{DnL17} that the normalized solution $\f$ to \eqref{eq:geom} is locally bounded in
the complement of the pluripolar set $\{\p=-\infty\}$.
We extend this important result here to the  hermitian setting:

\begin{thm} \label{thm:c0relative}
Assume $v_-(\omega)>0$ and fix $A>0$. 
Fix   $\mu=fdV_X$ a probability measure with
$f \in L^w(dV_X)$ and assume
$f \leq g e^{-A\p}$, where $g \in L^p(dV_X)$ for some $p>1$ and  $\p \in \PSH(X, \omega)$.
%for some convex increasing weight $w:\R^+ \rightarrow \R^+$ such that $w(t)/t \rightarrow +\infty$ as $t \rightarrow +\infty$. 
If $\f$ is a bounded $\omega$-psh function
such that    $(\omega+dd^c \f)^n=c f dV_X$ and $\sup_X \f=0$, then
$$
\p-\beta\leq \f \leq 0
$$
where $\beta>0$ depends on $p$, $w$ and upper bounds for $c,v_-(\omega)^{-1},||g||_{L^p}, \|f\|_w, A$.
\end{thm}

Proceeding by approximation, we are going to use in the sequel this lower bound
to ensure that some approximating solutions $\f_j$ are uniformly bounded from below
by a fixed function $\psi-\beta$. In particular any cluster point
$\f$ of the $\f_j$'s will be locally bounded in $X \setminus \{\p=-\infty\}$.
%Observe that if $\p \in PSH(X, \omega/a)$ then $\p \in PSH(X, \omega/a')$
%for any $0<a' \leq a$. The lower bound thus also yields  
%$\frac{a'}{2} \p-b'\leq \f$, with an improvement in $a'$ but a bigger additive contribution from $b' >b$.
Replacing $A$ and $\psi$ by $A/\alpha$ and $\alpha \psi$, with $0<\alpha \leq 1$, and increasing $\beta$
 in the statement, one obtains
\begin{equation*} %\label{eq:refinedrelative}
\alpha \p-C_{\alpha}\leq \f \leq 0.
\end{equation*}
An interesting consequence of this improved a priori estimate is that it produces limiting functions $\f$ having zero Lelong number at all points. 

\medskip

The proof of DiNezza-Lu   makes use of 
a theory of generalized Monge-Amp\`ere capacities further developed in \cite{DnL15}.
It seems quite delicate to extend their arguments to the hermitian setting, as they rely
on very precise integration by parts which produce several extra terms when $\omega$ is not closed.
Our new approach allows us to bypass these difficulties.
%, which again extends  directly to the hermitian setting.

  \begin{proof}
\noindent  {\it Uniform integrability of $\f$}.
%It is a standard measure theoretic fact that 
%The density $f$ belongs to an Orlicz class
%$L^w$ for some convex increasing weight $w:\R^+ \rightarrow \R^+$
%such that $w(t)/t \rightarrow +\infty$ as $t \rightarrow +\infty$.
Set $\chi_1(t):=-(w^*)^{-1}(-t)$, where $w^*$ denotes the Legendre transform of $w$.
Thus $\chi_1:\R^- \rightarrow \R^-$  is a convex increasing,
such that $\chi_1(-\infty)=-\infty$ and
$$
\int_X (-\chi_1 \circ \f) (\omega+dd^c \f)^n 
\leq c\int_X w \circ f dV_X+c\int_X (-\f) dV_X \leq C_0,
$$
as follows from the additive version of H\"older-Young inequality
and the compactness of $\sup$-normalized $\omega$-psh functions.
%When $X$ is K\"ahler the theory of finite energy classes allows to establish the inequality \eqref{eq:SKodaDnL2} below.
%This theory is not   available in the hermitian setting, so we use a different argument. 
We set, for $t \leq 0$,
$$
\chi(t)=-\int_t^0 \sqrt{(-\chi_1)(s)} ds
\; \; \text{ and } \; \;
\chi_2(t)=-\sqrt{(-\chi_1)} \circ \chi^{-1}(t).
$$
The weight $\chi$ is concave increasing with $\chi(-\infty)=-\infty$ and $\chi'(-\infty)=+\infty$,
while $\chi_2$ is convex increasing with $\chi_2(-\infty)=-\infty$.
Set $\Phi=P_{\omega}(\chi \circ \f) \leq \chi \circ \f \leq 0$.
It follows from Lemma \ref{lem:GLkey} that
\begin{eqnarray*}
\int_X (-\chi_2) \circ \Phi \,  (\omega+dd^c \Phi)^n 
&\leq&  \int_X (-\chi_2 \circ \chi \circ \f) \, \chi' \circ \f \, (\omega+dd^c \f)^n  \\
&=&  \int_X (-\chi_1 \circ \f) (\omega+dd^c \f)^n  \leq C_0.
\end{eqnarray*}
 Since $v_-(\omega)>0$ this provides a uniform bound on $\sup_X \Phi$ depending on $C_0$, $v_-(\omega)$ and $\chi_2$, so that
 $\Phi$ belongs to a compact family of $\omega$-psh functions.
 It follows from Skoda's uniform integrability result 
 that there exist uniform constants $\alpha,C>0$ such that
 $$
 \int_X \exp(-\alpha \chi \circ \f) dV_X \leq \int_X \exp(-\alpha \Phi) dV_X \leq C <+\infty.
 $$
 Finally since $\chi(t)$ grows faster than $t$ at infinity, we infer that for any $\lambda >0$, 
 \begin{equation} \label{eq:SKodaDnL2}
\int_X \exp(- \lambda \f) dV_X \leq M_{\lambda}
\end{equation}
for a uniform constant $M_{\lambda}$ which depends  on  upper bounds for $||f||_w$ and $v_-(\omega)^{-1}$.

    \smallskip
 
   \noindent  {\it Quasi-psh envelope}.
 %  The rest of the proof is identical to that of \cite[Theorem 3.3]{GL21a}.
   Consider $u:=P_{\omega}(2\f-\p)$. 
   One can show that $\sup_X u \leq C_0$ is uniformly bounded from above by generalizing \cite[Theorem 9.17.1]{GZbook}
   to the hermitian setting.  The previous argument also shows that $\sup_X P_{\omega}(2\f)$ is uniformly bounded from below, hence the same holds for $P_{\omega}(2\f-\psi)$ as we can assume without loss of generality that $\psi \leq 0$.  
   
   Since $\varphi$ is bounded, we have $u= P_{\omega}(2\varphi -\max(\psi,-t))$ for $t>0$ large enough. We can thus assume that $\psi$ is also bounded. 
   Since $2\f-\p$ is bounded and quasi continuous, it follows from Theorem \ref{thm:orthog}
   that the Monge-Amp\`ere measure
   $(\omega+dd^c u)^n$ is concentrated on the contact set 
   ${\mathcal C}=\{u=2\f-\p\}$. We claim that 
     $$
   (\omega+dd^c u)^n   \leq 1_{{\mathcal C}} 2^n (\omega+dd^c \f)^n.
   $$
   Indeed consider $v=u+\p$. This is a $(2\omega)$-psh function such that
   $v \leq 2\f$; it follows therefore from Lemma \ref{lem:Demkey} that
   $$
   1_{\{v=2\f \}} (2\omega+dd^c v)^n \leq 1_{\{v=2\f \}} (2\omega+dd^c (2\f))^n
   $$
   The claim follows since $\{v=2\f \}={\mathcal C}$ and $(\omega+dd^c u)^n \leq (2\omega+dd^c v)^n$.  Here the boundedness of $\psi$ ensures that the Monge-Amp\`ere measure $(\omega+dd^c v)^n$ is well-defined. 
   Now $(\omega+dd^c \f)^n=c fdV_X \leq cge^{-A\p}dV_X$, with
   $$
   1_{{\mathcal C}} g e^{-A\p}=1_{{\mathcal C}} g e^{Au} e^{-2A\f}
   \leq e^{A \sup_X u} ge^{-2A\f}.
   $$

  Using H\"older inequality and \eqref{eq:SKodaDnL2}, we infer that $g'=ge^{-2A\f} \in L^{p'}$, for some $p'>1$
   arbitrarily close to $p$, and with an upper bound on   $||g'||_{L^{p'}}$  that only depends
   on  $w$ and upper bounds for $A,||g||_{L^p}$, and $v_-(\omega)^{-1}$.
   
   Since $(\omega+dd^c u)^n \leq 2^n c e^{AC_0} g' dV_X$,
   we can invoke Theorem \ref{thm:uniformhermitian} to ensure
    that the oscillation of $u$ is uniformly bounded.
  The desired lower bound follows since
 $2\f = (2\f-\p) +\p \geq u+\p$.
   \end{proof}

 \section{Existence of locally bounded solutions}
 
% Recall that $(X,\omega_X)$ is a compact hermitian manifold of   dimension $n$, and $\omega$ is a semi-positive $(1,1)$-form.
 We now show the existence of solutions to degenerate complex Monge-Amp\`ere equations
under various positivity assumptions on $\omega$.
We  first establish a general existence result for bounded solutions in section \ref{sec:globallybdd},
then treat the case of solutions that are merely locally bounded in a Zariski open set in 
section \ref{sec:bdd}.
Higher regularity of these solutions will be studied  in section \ref{sec:reg}.

\subsection{Bounded solutions} \label{sec:globallybdd}

 \subsubsection{The normalizing constant}

Fix $p>1$ and $0\leq f \in L^p(X,dV_X)$ with $\int_X fdV>0$.
It follows from  \cite[Theorem 0.1]{KN15} and \cite[Corollary 0.2]{N16} that
for each $\varepsilon>0$, 
   there exists a unique constant $c_{\varepsilon}=c(\omega+\varepsilon \omega_X,f) >0$ 
   such that  one can find $ u \in \PSH(X,\omega+\varepsilon\omega_X) \cap L^{\infty}(X)$ satisfying
\[
(\omega+\varepsilon \omega_X +dd^c u)^n =c_{\varepsilon} fdV.
\]

 Observe  that $\e \mapsto c_{\varepsilon}$ is increasing.
Indeed assume $\varepsilon>\varepsilon'$ and $u_{\varepsilon}, u_{\varepsilon'}$ are bounded solutions to the corresponding equations. Then 
\begin{eqnarray*}
(\omega+ \varepsilon \omega_X +dd^c u_{\varepsilon'})^n 
&\geq  & (\omega+ \varepsilon' \omega_X +dd^c u_{\varepsilon'})^n \\
&= & c_{\varepsilon'} fdV = \frac{c_{\varepsilon'}}{c_{\varepsilon}} (\omega+\varepsilon \omega_X +dd^c u_{\varepsilon})^n.
\end{eqnarray*}
It follows therefore from Corollary \ref{cor:unique} that $c_{\varepsilon}\geq c_{\varepsilon'}$. 
We can thus consider:

\begin{defi} \label{def:normalizing}
We set
$
c(\omega,f):= \lim_{\varepsilon \to 0} c_{\varepsilon}.
$
\end{defi}

This definition is independent of the choice of the hermitian metric $\omega_X$:

\begin{prop} \label{prop: MA constant}
	 %For each $f\in L^p(X,dV)$ with $p>1$ and $\int_X fdV>0$, 
	 The constant $c(\omega,f)$ does not depend on 
	 %the choice of the particular Hermitian metric 
	 $\omega_X$.  Moreover
	 $$
	 c(\omega,f)=\inf \{ c(\omega',f) \; : \;  \omega' \text{  hermitian form such that } \omega'>\omega \}.
	 $$
\end{prop}

\begin{proof}
	Fix $\omega_1, \omega_2$  two hermitian metrics on $X$. 
	Then $A^{-1}\omega_1 \leq \omega_2 \leq A \omega_1$ for some positive constant $A$. It follows that, for all $\varepsilon>0$, 
	\[
	(\omega+A^{-1}\varepsilon \omega_1) \leq  (\omega+ \varepsilon \omega_2) \leq  (\omega+ A\varepsilon \omega_1). 
	\]
	Corollary \ref{cor:unique} ensures that
	\[
	c(\omega+A^{-1}\varepsilon \omega_1, f) \leq c(\omega+\varepsilon \omega_2, f)  \leq c(\omega+A\varepsilon \omega_1, f). 
	\]
	Letting $\varepsilon\to 0$ yields  the conclusion. 
	The last assertion follows similarly.
\end{proof}

% \begin{definition}
% 	We say that $\omega$ satisfies condition $(P)$ if for any $q>1$ there exists $c_q(\omega)>0$ such that:  if $0\leq f\in L^q(X)$ with $\|f\|_q\leq 1$, then $c(\omega,f)\geq c_q(\omega)$. 
% \end{definition}

%Any hermitian form satisfies condition $(P)$ as shown in Theorem \ref{thm:globaluniformhermitian}. By Step 1 of the proof of Theorem \ref{thm: bounded sol semipositive} below, any big semipositive form satisfies condition $(P)$. 

\subsubsection{Existence of solutions}

\begin{lemma}
	\label{lem: subsol big}
	 Fix $p>1$ and assume  either $\omega$ is big or $v_{-}(\omega)>0$. Then there exists $0<c=c(p,\omega)$ such that for 
	 any $f\in L^p(dV_X)$ with $\int_X f^p dV_X \leq 1$, we can find $u\in \PSH(X,\omega)\cap L^{\infty}(X)$ such that 
\[
(\omega+dd^c u)^n \geq c fdV_X
\; \; \text{and} \; \;
 -1\leq u \leq 0.
\]
\end{lemma}

\begin{proof}
{\it We first assume $\omega$ is big.}
	We fix $\rho\in \PSH(X,\omega)$ with analytic singularities such that
	 $\omega+dd^c \rho \geq \delta \omega_X$ for some $\delta>0$,
	with $\sup_X \rho =-1$. 
	Since $\rho$ belongs to $L^r$ for any $r>1$,   H\"older inequality ensures
	that  $|\rho|^{2n} f  \in L^q(dV_X)$ for some $q>1$. 
	%By Ko{\l}oziej-Nguyen \cite[Theorem 0.1]{KN15} there
	It follows from  the first step 
	of the proof of Theorem \ref{thm:globaluniformhermitian} that there
	 exist a uniform $c_1>0$ and $v\in \PSH(X,\omega_X)\cap L^{\infty}(X)$ such that 
	$ \sup_X v =-1$ and 
	\[
	(\omega_X+dd^c v)^n \geq  c_1|\rho|^{2n} fdV_X.
	\]
	
	Observe that the function $\delta v+\rho$ is $\omega$-psh with 
	$\omega+dd^c (\delta v+\rho) \geq \delta (\omega_X+dd^c v)$.
	Set $u:= -(\delta v+\rho)^{-1}=\chi \circ (\delta v+\rho)$ with $\chi(t)=-t^{-1}$ convex increasing on $\R^-$.
	Our normalizations ensure  $-1 \leq u \leq 0$ and 
	 a direct computation   yields
	\[
	\omega+dd^c u \geq 
	\delta \chi' \circ (\delta v+\rho)   (\omega_X+dd^c v)  = \frac{\delta}{(\delta v+\rho)^{2}} (\omega_X +dd^c v).
	\]
	We infer $\omega_u^n \geq |\delta v+\rho|^{-2n} \delta^n c_1 |\rho|^{2n} fdV_X$.
	 Since $v \leq -1$ is bounded  and $\rho\leq -1$, it follows that $\omega_u^n \geq a fdV_X$ 
	 for some uniform constant $a>0$. 
	 
	 \smallskip

{\it We now assume $v_-(\omega) >0$.} 
Set $\omega_j=\omega+2^{-j} \omega_X$.
Rescaling $\omega_X$ if necessary we can assume that
$\omega \leq \omega_j \leq \omega_X$.
Using \cite[Theorem 0.1]{KN15} we solve 
\[
(\omega_j +dd^c u_j)^n = c_j (1+f) dV_X,\; \sup_X u_j=0
\]
for any fixed $j>0$, where $u_j \in \PSH(X,\omega_j) \cap L^{\infty}(X)$.
Fix $a>0$ such that $\omega_X^n \geq a dV_X$, and observe that
$
(\omega_X +dd^c u_j)^n \geq (\omega_j +dd^c u_j)^n \geq c_j dV_X,
$
hence
$$
(\omega_X +dd^c u_j) \wedge \omega_X^{n-1} \geq c_j^{\frac{1}{n}} a^{\frac{n-1}{n}} dV_X,
$$
as follows from Lemma \ref{lem:AGM}.

We infer that $c_j$ is uniformly bounded from above.
Since $||(1+f)||_p \leq 2$ and $v_-(\omega_j) \geq v_-(\omega)>0$,
it follows from Theorem \ref{thm:uniformhermitian}
that $u_j$ is uniformly bounded.

Extracting and relabelling we can assume that $u_j \rightarrow u \in \PSH(X,\omega)$.
%It follows from Hartogs lemma that $\sup_X u=0$, and
Theorem \ref{thm:uniformhermitian} yields $-T \leq u \leq 0$
for some uniform constant $T >0$.
 Set
$$
\Phi_j:=\left( \sup_{\ell \geq j} u_{\ell} \right)^*. 
$$
By Lemma \ref{lem:envmin}, we have 
\[
(\omega+2^{-j} \omega_X +dd^c \Phi_j)^n \geq v_{-}(\omega) fdV_X.
\]
Since $\Phi_j \searrow u\in \PSH(X,\omega) \cap L^{\infty}(X)$, we obtain $(\omega+dd^c u)^n \geq v_{-}(\omega) fdV_X$. 

If $0 \leq T \leq 1$ we are done, otherwise we consider $\tilde{u}=u/T \in \PSH(X,\omega)$ and observe that
$-1 \leq \tilde{u} \leq 0$ with $(\omega+dd^c \tilde{u})^n \geq T^{-n}v_{-}(\omega) fdV_X$.
\end{proof}

\begin{theorem}\label{thm: bounded sol semipositive}
	Let $\mu=f dV_X$ be a probability measure, where
	$0\leq f\in L^p(dV_X)$ for some $p>1$. 
		Assume either $\omega$ is   big or  $v_{-}(\omega)>0$.
		\begin{enumerate}
		\item For any $\lambda>0$ there exists a unique $\f_{\lambda} \in \PSH(X,\omega)\cap L^{\infty}(X)$ such that 
	\[
	(\omega+dd^c \f_{\lambda})^n = e^{\lambda \f_{\lambda}} f dV_X. 
	\]
	Moreover  $||\f_{\lambda}||_{\infty} \leq C$, where $C$ depends on upper bounds on $\|f\|_p$, $\|f\|_{1/n}^{-1},\lambda^{-1}$. 
	
	\item One has $c(\omega,f)>0$ and there exists $\varphi\in \PSH(X,\omega) \cap L^{\infty}(X)$ such that  
	\[
	(\omega+dd^c \varphi)^n = c(\omega,f)\mu. 
	\]
		\end{enumerate}			
\end{theorem}

Here $c(\omega,f)$ denotes the normalizing constant from Definition \ref{def:normalizing}. It is possible to study H\"older regularity of $\f$ following  \cite{DDGKPZ}, \cite{LPT20}, \cite{KN19}, but we leave it for a future project. 

%To prove Theorem \ref{thm: bounded sol semipositive} we will first solve the Aubin-Yau equation: 

%\begin{prop}\label{prop: bdd sol Aubin-Yau semi pos}
%	Assume the same setting as in Theorem \ref{thm: bounded sol semipositive} and 
%\end{prop}

\begin{proof}
{\it We first prove (1).}
	For simplicity we  assume $\lambda=1$. 
		By \cite[Theorem 0.1]{N16}, for each $\varepsilon>0$
		there exists $\varphi_{\varepsilon} \in \PSH(X,\omega+\varepsilon \omega_X) \cap L^{\infty}(X)$ 
		such that
	\[
	(\omega+\varepsilon \omega_X +dd^c \varphi_{\varepsilon})^n = e^{\varphi_{\varepsilon}} \mu. 
	\]
	The bounded subsolution $u$ obtained in Lemma \ref{lem: subsol big} satisfies 
	\[
	(\omega+\varepsilon \omega_X +dd^c u)^n \geq (\omega +dd^c u)^n
	\geq e^{u  +\ln c} \mu,
	\]
	since $u \leq 0$.
	The domination principle now yields $u+ \ln c \leq \varphi_{\varepsilon}$, for all $\varepsilon>0$.
	
	It follows again from 
	 the domination principle  that $\varphi_{\varepsilon}$ decreases to some function $\varphi \in \PSH(X,\omega)$
	 which is bounded below by the previous estimate, and bounded above by $\f_1$, hence
	 it is uniformly bounded.
	 
	 The uniqueness is a simple consequence of Corollary \ref{cor:unique2}.

%We first treat the case when $\omega$ is big. 
%The other case $v_{-}(\omega)>0$ follows from similar arguments.   
%We proceed in three steps. 

\medskip

%\noindent {\it Step 1: constructing bounded subsolutions}. 
%	It follows from Lemma \ref{lem: subsol big} that  there exists $a>0$ and $u\in PSH(X,\omega) \cap L^{\infty}(X)$ such that 
%	\[
%	(\omega+dd^c u)^n \geq a \mu. 
%	\]
%	
%	Recall that $\rho\in PSH(X,\omega)$ has analytic singularities and $\omega+dd^c \rho \geq \delta \omega_X$ with $\delta>0$. 
%	We normalize $\rho$ by $\sup_X \rho =-1$. 
%	Since $\rho$ belongs to $L^r$ for any $r>1$, it follows from H\"older inequality
%	that  $|\rho|^{2n} f  \in L^q(X)$ for some $q>1$. 
%	%By Ko{\l}oziej-Nguyen \cite[Theorem 0.1]{KN15} there
%	It follows from  the first step 
%	of the proof of Theorem \ref{thm:globaluniformhermitian} that there
%	 exist a uniform $c_1>0$ and $v\in PSH(X,\omega_X)\cap L^{\infty}(X)$ such that 
%	$ \sup_X v =-1$ and 
%	\[
%	(\omega_X+dd^c v)^n \geq  c_1|\rho|^{2n} \mu.
%	\]
%	
%	Observe that the function $\delta v+\rho$ is $\omega$-psh with 
%	$\omega+dd^c (\delta v+\rho) \geq \delta (\omega_X+dd^c v)$.
%	Set $u:= -(\delta v+\rho)^{-1}=\chi \circ (\delta v+\rho)$ with $\chi(t)=-t^{-1}$ convex increasing on $\R^-$.
%	 A direct computation thus yields
%	\[
%	\omega+dd^c u \geq 
%	\delta \chi' \circ (\delta v+\rho)   (\omega_X+dd^c v)  = \frac{\delta}{(\delta v+\rho)^{2}} (\omega_X +dd^c v).
%	\]
%	We infer $\omega_u^n \geq |\delta v+\rho|^{-2n} \delta^n c_1 |\rho|^{2n} \mu$.
%	 Since $v$ is bounded  and $\rho\leq -1$, it follows that $\omega_u^n \geq a \mu$ for some $a>0$. 
%	
%	\medskip
	
%	\noindent {\it Step 2: a priori $L^{\infty}$ estimates}. 

\noindent{\it We now prove (2).}
	It follows from \cite[Theorem 0.1]{KN15,KN19} that
	there exists $c_j>0$ and $\f_j \in \PSH(X,\omega) \cap L^{\infty}(X)$ such that $\sup_X \varphi_j=-1$ and
	\[
	(\omega_j +dd^c \varphi_j)^n = c_j \mu, 
	\]
	where $\omega_j:=\omega+ 2^{-j} \omega_X$.  
	By Lemma \ref{lem: subsol big} we have a lower bound $c_j \geq a$, while the domination principle ensures that $j \mapsto c_j $ is decreasing,
so $c_j \rightarrow c(\omega,f) \in [a,c_1]$.
	
	Fix $\gamma>0$ so small that $g_j:=e^{-\gamma \varphi_j} f$ is uniformly bounded in $L^q(X)$ for some $q>1$. This is possible thanks to the Skoda integrability theorem \cite{Sko72} and H\"older inequality. 
	By the previous step the solutions $u_j\in \PSH(X,\omega)$ to 
	\[
	(\omega+dd^c u_j)^n = e^{\gamma u_j} g_j dV_X
	\]
are uniformly bounded. Since $u_j$ satisfies 
	\[
	(\omega+ 2^{-j} \omega_X +dd^c u_j)^n \geq e^{\gamma (u_j- \gamma^{-1}\ln c_1)} c_1g_jdV_X,
	\]
	while $\varphi_j$ satisfies 
	\[
	(\omega+ 2^{-j} \omega_X +dd^c \varphi_j)^n \leq  e^{\gamma \varphi_j} c_1g_jdV_X,
	\]
	the domination principle ensures that $\varphi_j \geq u_j - \gamma^{-1} \ln c_1 \geq -C$. 
	
		\medskip
	
Extracting a subsequence we can assume that $\f_j$ converges in $L^1$ and a.e. to $\f\in \PSH(X,\omega) \cap L^{\infty}(X)$. 
	Set
$$
\Phi_j^+:=\left( \sup_{\ell \geq j} \f_{\ell} \right)^*
\; \; \text{ and } \; \; 
\Phi_j^-:=P_{\omega}\left( \inf_{\ell \geq j} \f_{\ell} \right).
$$
The function $\Phi_j^+$ (resp. $\Phi_j^-$) is $\omega_j$-psh (resp. $\omega$-psh) and uniformly bounded on $X$. 
 It follows from Lemma \ref{lem:envmin} that
 $$
(\omega_j+dd^c \Phi_j^+)^n \geq c_j^- fdV_X 
\; \; \text{ with }  \; \;
c_j^-:=\inf_{ \ell \geq j} c_{\ell}.
 $$
 Since $\Phi_j^+$ decreases to $\f$, while $c_j^-$ increases to $c$, it follows from the continuity
 of the complex Monge-Amp\`ere along monotonic sequences (see \cite{BT76,BT82}) that
 $$
(\omega+ dd^c \f)^n \geq c f dV_X
\; \; \text { on } \; \;
X.
 $$
Writing $(\omega_j +dd^c \f_j)^n = e^{\f_j-\f_j} c_j fdV_X$, it follows  from Lemma \ref{lem:envmin} that
 $$
(\omega+ dd^c \Phi_j^-)^n \leq e^{\Phi_j^--\inf_{\ell \geq j} \f_{\ell}} c_j^+ f dV_X
\; \; \text{ with }  \; \;
c_j^+ := \sup_{\ell \geq j} c_{\ell}.
 $$
The functions $\Phi_j^-$ increase to some $\varphi' \in \PSH(X,\omega)\cap L^{\infty}(X)$ 
which satisfies 
$$
(\omega+dd^c \varphi')^n \leq e^{\f'-\f}cfdV_X \; \text{on}\; X.
$$
The domination principle then yields $\f' \geq \f$ but we also have by construction that $\f'\leq \f$, hence $\f=\f'$ is a solution. 
% but we do not know whether $\varphi'=\varphi$. 
% To go around this difficulty we use the following detour. 
%For each $\beta>0$ small we solve 
%\[
%(\omega+dd^c \phi_{\beta})^n = e^{\beta( \phi_{\beta}- \varphi)} c \mu , 
%\]
%with $ \phi_{\beta} \in PSH(X,\omega) \cap L^{\infty}(X)$.
%The existence of $\phi_{\beta}$ follows from Step 2. The domination principle yields
%\[
%\varphi' + \sup_X (\varphi- \varphi') \geq \phi_{\beta'} \geq \phi_{\beta} \geq \varphi, 
%\]
% if $\beta > \beta' >0$. Letting $\beta \to 0$, the decreasing limit $\phi \in PSH(X,\omega)\cap L^{\infty}(X)$  of the 
% $\phi_{\beta}$'s  is a bounded solution
%  to $(\omega+dd^c \phi)^n = c \mu$. 	
 %We finally treat the case when $v_-(\omega)>0$.
% The proof is similar, replacing Lemma \ref{lem: subsol big} by Theorem \ref{thm:uniformhermitian}
% and noting that 
% $$
% c_j=c(\omega+2^{-j}\omega_X,f) =\int_X (\omega+2^{-j} \omega_X +dd^c \f_j)^n  \geq v_-(\omega)>0,
% $$
% which yields in particular $c(\omega,f) \geq v_-(\omega)>0$ in this case.
\end{proof}

\subsection{Uniqueness of solutions}

%We have established  an $L^1-L^{\infty}$-stability estimate when $v_-(\omega)>0$ (see Theorem \ref{thm:stabilityhermitian}),
%which allows e.g. to show uniform convergence of solutions as soon as they converge in $L^1$. 
 We now study the dependence of the solution with respect to the density, dealing as well with the case
when the  RHS  depends on the solution in a convex way.

   \begin{theorem}  \label{thm: stab sem pos}
  Assume $\omega$ is  either big
   or such that $v_-(\omega)>0$.
   %or such that $v_-(\omega)>0$.
  Fix $f_1,f_2 \in L^p(dV_X)$ with $p>1$ and 
  $A^{-1} \leq \left(\int_X f_i^{\frac{1}{n}} dV_X\right)^n \leq \left(\int_X f_i^p dV_X \right)^{\frac{1}{p}}\leq A$,
  for some  constant $A>1$. 
  Assume  $\f_1,\f_2 \in \PSH(X,\omega) \cap L^{\infty}(X)$ satisfy
      $$
      (\omega+dd^c \f_i)^n=e^{\lambda \f_i}f_idV_X.
      $$ 
  
  \begin{enumerate}
  \item    If $\lambda >0$, then    
 $
||\f_1-\f_2||_{\infty} \leq  T  ||f_1-f_2||^{\frac{1}{n}},
 $
 where $T$  is a   constant which depends on $n,p$ and upper bounds for  $A$, $\lambda^{-1}, \lambda$. 
 
 \smallskip
 
 \item If $\lambda=0$ and $f_1 \geq c_0>0$, then the same estimate holds with $T$ depending on $n,p$ and upper bounds for  $A, c_0^{-1}$.
  \end{enumerate}
  \end{theorem}
  
      In particular, there is at most one bounded $\omega$-psh solution $\varphi$ to the equation
      $(\omega+dd^c \f)^n=e^{\lambda \f}f \, dV_X$ in case
      $\lambda>0$, or when $\lambda=0$ and $f\geq c_0>0$.
      
      \smallskip
 
 When $\omega$ is hermitian such a stability estimate has been provided by Ko{\l}odziej-Nguyen 
 for $\lambda=0$ 
 \cite{KN19}.
 %(with refinements by \cite{LPT20}) 
 We   adapt some arguments of
 Lu-Phung-T\^o  \cite{LPT20}, \cite{GLZ18}, who treated the case $\lambda>0$,
 and obtained refined estimates in the case $\lambda=0$.
 %inspired by Guedj-Lu-Zeriahi \cite{GLZ18}. 

 \begin{proof}
(1)  We first consider the simpler case when $\lambda>0$.
 	Rescaling we can assume without loss of generality that $\lambda=1$.  
 	We first observe that by 
 	Theorem \ref{thm: bounded sol semipositive}
 	 $\f_1,\f_2$ are uniformly bounded. Since $\f_1, \f_2$ play the same role, it suffices to show that $\f_1 - \f_2\leq T \|f_1-f_2\|_p^{1/n}$. Using Theorem \ref{thm: bounded sol semipositive} we solve 
 	\[
 	(\omega+dd^c u)^n = c \left ( \frac{|f_1-f_2|}{\|f_1-f_2\|_p^{1/n}} +1\right ) dV_X,
 	\]
 	with $\sup_X u =0$ and
 	$|u|, c, c^{-1}$ uniformly bounded. We can assume that $$\e := e^{(\sup_X u -\log c)/n}\|f_1-f_2\|_p^{1/n}$$ is small enough. We next consider 
 	\[
 	\phi := (1-\e) \f_1 +\e u -C\varepsilon,
 	\]
 	and proceed exactly the same as in \cite[Theorem 1.1]{LPT20} to prove that 
 	$$(\omega+dd^c \phi)^n \geq e^{\phi} f_2 dV_X.
 	$$ 
 	The domination principle thus gives $\phi \leq \f_2$, yielding the desired estimate.

 \bigskip
 
 (2) We now take care of the case $\lambda=0$.
 As usual we use $C$ to denote various uniform constants. 
 	 By Theorem \ref{thm: bounded sol semipositive} we have that $\sup_X (|\f_1|+|\f_2|) \leq C$. 
 	 
 	 \smallskip
 	 
 	 \noindent
 {\it Step 1.} We first assume that  $0<f_1,f_2$ are smooth,  and 
 	\begin{equation}
 		\label{eq: stab assump}
 			2^{-\varepsilon} f_1 \leq f_2 \leq 2^{\varepsilon} f_1,
 	\end{equation}
 	for some small constant $\varepsilon >0$. Our goal is to show that $|\f_1-\f_2|\leq C\e$.

 %	Adding a constant to $u$, we can assume that $\sup_X (u-v) = \sup_X (v-u) =s>0$ and the goal is to prove that $s\leq C\varepsilon$. 
 Let $m=\inf_X (\f_1-\f_2)$ and $M=\sup_X (\f_1-\f_2)$. 
 To simplify the notation we set $\mu= f_1^pdV_X$.  Recall that by Lemma \ref{lem: subsol big} there exists a constant $c=c(\omega,p)>0$ such that if $0\leq f \in L^p(X,dV_X)$ with $\|f\|_p\leq 1$ then we can find $u\in \PSH(X,\omega)$ such that $(\omega+dd^c u)^n \geq 4c fdV_X$ with $-1\leq u \leq 0$. 
   We define 
 	\[
 	t_0:= \sup \{t \in  [m,M] \; : \; \mu(\f_1 < \f_2+ t) \leq c^p\},
 	\]
 	\[
 	T_0:= \inf \{t \in [m,M] \; : \; \mu(\f_1 > \f_2 + t) \leq c^p\}. 
 	\] 

\medskip 

\noindent {\it Step 1.1.} We claim that $t_0\leq m+C\e$ and $T_0 \geq M-C\e$. 

 	Assume $t\leq t_0$ and $\mu(u<v+t) \leq c^p$. Set	
 	\[
	\hat{f} := \frac{1}{2c} {\bf 1}_{\{\f_1 < \f_2+ t\}}f_1 +\frac{1}{2\|f_1\|_p}{\bf 1}_{\{\f_1\geq \f_2+t\}}f_1.
	\]
	It follows from Lemma \ref{lem: subsol big} that $c(\omega,\hat{f}) \geq 4c$, since
	\begin{flalign*}
		\int_X \hat{f}^p dV_X &\leq  \frac{1}{(2c)^p}\int_{\{\f_1<\f_2+t\}} f^p_1 dV + 2^{-p} \leq 1.
	\end{flalign*}
	 Using Theorem \ref{thm: bounded sol semipositive} we  find $\phi \in \PSH(X,\omega) \cap L^{\infty}(X)$ such that
	 $\sup_X (\phi-\f_2)=0$ and
 	\[
 	(\omega+dd^c \phi)^n = c(\omega,\hat{f}) \hat{f}dV_X.
 	\]
 	Since $\hat{f} \geq C^{-1} f_1$, it follows that $4c\leq c(\omega,\hat{f})\leq C$, hence 
 	 $\phi$ is uniformly bounded. 
 	
 	We next consider $\psi:= (1-\varepsilon)\f_2 + \varepsilon \phi +t$. Then  $\{\f_1<\psi\}\subset \{\f_1<\f_2+t\}$ and on the latter set we have 
 	\[
 	\omega_{\psi}^n \geq ((1-\varepsilon)2^{-\varepsilon/n} + 2^{1/n}\varepsilon )^n f_1dV \geq (1+\gamma) f_1dV = (1+\gamma) \omega_{\f_1}^n,  
 	\]
 	for some $\gamma>0$. The domination principle ensures that $\f_1\geq \psi \geq \f_2 +t -C\varepsilon$, 
 	hence $t\leq \f_1-\f_2+ C\varepsilon$. Since this inequality is true for all $x\in X$, we can take $x$ such that $\f_1(x)-\f_2(x) =\inf_X (\f_1-\f_2)$ and obtain   
 	\begin{equation}
 		\label{eq: stab t0}
 		t_0\leq m+ C\varepsilon.
 	\end{equation} 
 	Using a similar argument we  obtain
 	\begin{equation}
 		\label{eq: stab T0}
 		T_0\geq M- C\varepsilon.
 	\end{equation}

\smallskip

\noindent {\it Step 1.2.} 
 	Fixing $t_0<a <b<T_0$, we claim that
 	\begin{equation}
 		\label{eq: stab 1}
 		I_{a,b}:=\int_{E_{a,b}} d(\f_1-\f_2) \wedge d^c (\f_1-\f_2) \wedge \omega_{\f_1}^{n-1} \leq C(b-a)(\varepsilon +b-t_0),
 	\end{equation}
 	and 
 	\begin{equation}
 		\label{eq: stab 2}
 		I_{a,b}\int_{E_{a,b}} \omega_{\f_1} \wedge \omega_X^{n-1} \geq C^{-1}  \left (\int_{E_{a,b}} |\partial (\f_1-\f_2)| dV_X\right )^2,
 	\end{equation}
 	where $E_{a,b}:= \{\f_2+a < \f_1 < \f_2+b\}$.

 	To prove \eqref{eq: stab 1} we set $\varphi:= \max(\f_1,\f_2+a)$, $\psi:= \max(\f_1,\f_2+b)$, and consider
 	\[
 	J:= \int_X (\varphi-\psi)(\omega_{\psi}^n - \omega_{\varphi}^n)= \int_X (\varphi-\psi) dd^c (\psi-\varphi) \wedge T,
 	\]
 	where $T= \sum_{j=0}^{n-1} \omega_{\varphi}^j \wedge \omega_{\psi}^{n-1-j} \geq \omega_{\varphi}^{n-1}$.
 	Integrating by parts we obtain 
 	\begin{flalign*}
 		J & = \int_X d(\varphi-\psi) \wedge d^c (\varphi-\psi) \wedge T - \frac{1}{2}\int_X (\varphi-\psi)^2 dd^c T\\
 		& \geq \int_{E_{a,b}} d(\f_1-\f_2) \wedge d^c (\f_1-\f_2) \wedge \omega_{\f_1}^{n-1} - C (b-a)^2. 
 	\end{flalign*}
 	 We have used here $0\leq \psi-\varphi \leq b-a$, $\omega_{\f_1}^{n-1}= \omega_{\varphi}^{n-1}$ on $\{\f_1>\f_2+a\}$, and 
 	 \[
 	 dd^c T \geq -C (\omega^2\wedge (\omega_{\f_1}+\omega_{\f_2})^{n-2}+\omega^3\wedge (\omega_{\f_1}+\omega_{\f_2})^{n-3}),
 	 \]  
 	 the integral over $X$ of the latter form being uniformly bounded. 
 	 We have also used that $\partial \psi = 0$ a.e. on $\{\f_1\geq \f_2+b\} \cup \{\f_1\leq \f_2+a\}$.

 	 On the other hand, $(\varphi-\psi)(\omega_{\varphi}^n-\omega_{\psi}^n) =0$ on $\{\f_1<\f_2+a\} \cup \{\f_1\geq \f_2+b\}$, and 
 	 \[
 	 \left | \int_{E_{a,b}}(\varphi-\psi)(\omega_{\varphi}^n-\omega_{\psi}^n)   \right |  \leq  (b-a) \int_X |f-g|dV  \leq C(b-a)\varepsilon. 
 	 \] 
 	 Now
 	 \begin{flalign*}
 	 	 \left | \int_{\{\f_1=\f_2+a\}} (\varphi-\psi)(\omega_{\psi}^n - \omega_{\varphi}^n) \right | &= (b-a) \left | \int_{\{\f_1=\f_2+a\}} (\omega_{\varphi}^n - \omega_{\f_2}^n) \right | \\
 	 	& = (b-a) \left |\int_{\{\f_1\leq \f_2+a\}} (\omega_{\varphi}^n- \omega_{\f_2}^n)\right | \\
 	 	& \leq   C(b-a) \varepsilon  + (b-a) \left |\int_{\{\f_1\leq \f_2+a\}} (\omega_{\varphi}^n- \omega_{\f_1}^n)\right | \\
 	 	& = C(b-a) \varepsilon  +  (b-a)\left |\int_X (\omega_{\varphi}^n- \omega_{\f_1}^n)\right | \\
 	 	& \leq C(b-a) \varepsilon + (b-a)  C\sup_X (\varphi-\f_1)\\
 	 	&\leq  C(b-a) (\varepsilon + a-t_0).
 	 \end{flalign*}
 	 These inequalities are justified as follows:
 	 \begin{itemize}
 	 \item  in the second line we used the identity $\omega_{\varphi}^n = \omega_{\f_2}^n$ on $\{\f_1<\f_2+a\}$. 
 	 \item in the third line we used 
 	 $
 	\left |  \int_{\{\f_1\leq \f_2+a\}} (\omega_{\f_1}^n - \omega_{\f_2}^n)  \right | \leq \int_X |f-g| dV_X \leq C\varepsilon.  
 	 $
 	 \item in the fourth line we used the identity $\omega_{\varphi}^n= \omega_{\f_1}^n$ on $\{\f_1>\f_2+a\}$.
 	 \item the fifth line consists in the following  integration by part,
 	 \[
 	 \left | \int_X (\omega_{\varphi}^n - \omega_{\f_1}^n) \right | = \left | \int_X (\varphi-\f_1) dd^c  \sum_{j=0}^{n-1} \omega_{\varphi}^j \wedge \omega_{\f_1}^{n-1-j} \right | \leq C \sup_X (\varphi-\f_1). 
 	 \]	
 	 \end{itemize}
 Altogether this yields \eqref{eq: stab 1}, as
 	 \[
 	 I_{a,b} \leq C(b-a) (\varepsilon +  a -t_0+ b-a) \leq C_1(b-a)(b-t_0+\varepsilon).
 	 \]

 	 We next prove \eqref{eq: stab 2}. 
 	 The following pointwise   inequality is a reformulation of  \cite[Lemma 3.1]{Pop16}:
 	  if $\alpha\geq 0$ is a $(1,1)$-form and $\omega_1,\omega_2$ are  Hermitian forms then 
 	 \[
 	 \frac{\alpha \wedge \omega_1^{n-1}}{\omega_2^n} \frac{\omega_1 \wedge \omega_2^{n-1}}{\omega_2^n} \geq \frac{1}{n} \frac{\omega_1^n}{\omega_2^n}  \frac{\alpha \wedge \omega_2^{n-1}}{\omega_2^n}. 
 	 \] 
 	 Applying this  for $\alpha = d(\f_1-\f_2) \wedge d^c (\f_1-\f_2)$, $\omega_1= \omega_{\f_1}$, $\omega_2=\omega_X$ we obtain
 	 \begin{flalign*}
 	 	\frac{d(\f_1-\f_2) \wedge d^c (\f_1-\f_2) \wedge \omega_{\f_1}^{n-1}}{\omega_X^n} \frac{\omega_{\f_1} \wedge \omega_X^{n-1}}{\omega_X^n} & \geq  \frac{1}{n} \frac{\omega_{\f_1}^n}{\omega_X^n} |\partial (\f_1-\f_2)|^2\\
 	 	& \geq \frac{c_0}{nC} |\partial(\f_1-\f_2)|^2.
 	 \end{flalign*}
 	 Applying  Cauchy-Schwarz inequality we obtain \eqref{eq: stab 2}.
 	  Let us stress that this is the only place where the extra assumption $f_1\geq c_0$ is used.  

\smallskip

 \noindent {\it Step 1.3.} 
 %To finish the proof of Step 1, 
 We finally set $t_1=t_0+\varepsilon$, $t_k:= t_0+ 2^{k-1}(t_1-t_0)$, for $k\geq 2$, 
 so  that $t_{k+1}-t_k=2^{k-1}\varepsilon$. Fix $k$ such that $t_0 \leq t_k <t_{k+1} \leq T_0$.  Then 
 	$$
 	\mu(\f_1<\f_2+t_k) \geq c^p \; \text{and}\;  \mu(\f_1>\f_2+t_{k+1}) \geq c^p,
 	$$ 
  hence  
  %the Sobolev-Poincar\'e inequality, 
  Lemma \ref{lem: KN19 Sobolev}  below (with $h= \f_1-\f_2-t_k$ and $\delta=t_{k+1}-t_k$) yields
 	\begin{equation*}% \label{eq: Sobolev}
 			\int_{\{\f_2+t_{k}<\f_1\leq \f_2+t_{k+1}\}} |\partial (\f_1-\f_2)| dV_X \geq C^{-1} (t_{k+1}-t_k)  =  C^{-1} 2^{k-1}\varepsilon.
 	\end{equation*} 
 	 Combining this with \eqref{eq: stab 1} and \eqref{eq: stab 2} for $a=t_k$ and $b=t_{k+1}$ we obtain 
 	 %the following estimate 
 	 \begin{flalign} 
 	 		 & C2^{2k} \varepsilon^2\int_{E_{a,b}} \omega_{\f_1} \wedge \omega_X^{n-1} \nonumber  \geq  I_{a,b}\int_{E_{a,b}} \omega_{\f_1} \wedge \omega_X^{n-1} \geq  C^{-1}2^{2k} \varepsilon^2,\label{eq: grad on small collars}
 	 \end{flalign}
hence
 	 \begin{flalign*}
 	 	 \int_{\{\f_2+t_k <\f_1\leq \f_2+t_{k+1}\}} \omega_{\f_1} \wedge \omega_X^{n-1}  \geq C^{-1}. 
 	 \end{flalign*}
 	We take $N$ such that $t_N<T_0\leq t_{N+1}$.  Summing up for $k$ from $1$ to $N$ we obtain 
 	$$
 	\int_{\{\f_2+t_1 <\f_1\leq \f_2+t_{N}\}} \omega_{\f_1} \wedge \omega_X^{n-1} \geq NC^{-1}.
 	$$ 
 	Since the integral $\int_X \omega_{\f_1} \wedge \omega_X^{n-1}$ is uniformly bounded, 
 	this yields a uniform upper bound for $N$. Using \eqref{eq: stab t0} and \eqref{eq: stab T0} we thus obtain
 	the desired bound
 	 \[
 	M-m =  M-T_0 + T_0-t_{N+1} + \sum_{k=1}^{N+1} (t_{k}-t_{k-1}) +t_0-m  \leq C\varepsilon. 
 	 \]

\medskip

\noindent {\it Step 2.} 
To remove the assumption \eqref{eq: stab assump} on $f_1,f_2$ from Step 1 one can proceed as in \cite{LPT20} 
to which we refer the reader (this requires the case $\lambda=1$ which has been treated independently above). 

We finally remove the smoothness assumption. Let $f_{1,j}, f_{2,j}$ be sequences of smooth positive densities converging to $f_1,f_2$ respectively in $L^p(X,dV_X)$. Let also $\f_{1,j}$, $\f_{2,j}$ be smooth functions decreasing to $\f_1,\f_2$ respectively. 
Using \cite{Cher87} we solve, for $i=1,2$,
\[
(\omega_j+dd^c u_{i,j})^n = e^{u_{i,j}-\f_{i,j}} f_{i,j} dV_X, 
\]
where $\omega_j=\omega+ 2^{-j} \omega_X$. 
Theorem \ref{thm: bounded sol semipositive} ensures that the functions  $u_{i,j}$ are uniformly bounded. Note that the constants $c(p, \omega_j)$ from Lemma \ref{lem: subsol big} satisfy  $c(p,\omega_j) \geq c(p, \omega)>0$. By Step 1 we have 
\[
|u_{1,j}-u_{2,j}| \leq T \|g_{1,j}-g_{2,j}\|_p^{1/n},
\]
where $g_{i,j}= e^{u_{i,j}-\f_{i,j}} f_{i,j}$. As $j \to +\infty$, arguing as in the proof of Theorem \ref{thm: bounded sol semipositive} we can show that $u_{i,j}$ converge to $u_i$ solving 
\[
(\omega+dd^c u_i)^n = e^{u_i-\f_i} f_i dV_X.
\]
The domination principle ensures that $u_i=\f_i$, hence $\|g_{1,j}-g_{2,j}\|_p \to \|f_1-f_2\|_p$ 
finishing the proof. 
\end{proof}

 We have used the following elementary result (see \cite[Lemma 2.6]{KN19}):

\begin{lemma}\label{lem: KN19 Sobolev}
	Let $h$ be a real-valued function in $W^{1,1}(X)$ such that 
	\[
	{\rm Vol}(h \leq 0) \geq \gamma \; \; \text{and}\; \;  {\rm Vol}(h \geq \delta ) \geq \gamma
	\] 
	where $\delta>0$ and $\gamma>0$ are constants. Then 
	\[
	\int_{\{0 < h <  \delta \}} |\partial h| dV_X \geq C^{-1}  \gamma^{\frac{2n-1}{2n}} \delta. 
	\]
\end{lemma}

Here $W^{1,1}(X)$ denotes the set of $L^1$  functions whose gradient
belongs to $L^1$.

%\begin{proof}
%We first recall the following Sobolev-Poincar\'e inequality (see \cite[Proposition 3.9]{Heb96}): 
%	there exists a constant $C=C(X,n)$ such that for all $u\in W^{1,1}(X)$ 	 with $\int_X u dV_X=0$, we have 
%\[
%\left ( \int_X |u|^{\frac{2n}{2n-1}} dV_X \right)^{\frac{2n-1}{2n}} \leq  C\int_X |\nabla u| dV_X.
%\]

%	We set $\phi := \max(h,0) - \max(h-\delta,0) -M$, where $M$ is a constant such that $\int_X \phi dV=0$. 
%	Setting $r= 2n/(2n-1)$,  the Sobolev-Poincar\'e inequality yields
%	\[
%	\left (\int_X |\phi|^r  dV_X \right)^{1/r} \leq C \int_X |\partial \phi | dV_X. 
%	\]
%	Observe that $\phi= -M$ on $\{h \leq 0\}$ and $\phi= \delta -M$ on $\{h \geq \delta \}$. It follows from \cite[Lemma 7.6, 7.7]{GT01} that $\partial \phi =0$ on $\{h \leq 0\} \cup \{h \geq \delta \}$ and $\partial \phi = \partial h$ on $\{0<h<\delta\}$, hence 
%	\[
%	C\int_{\{0<h < \delta \}} |\partial h | dV_X \geq  \left ( |M|^r \int_{\{h \leq 0\}} dV + |\delta -M|^r \int_{\{h \geq \delta \}} dV_X \right )^{1/r}%\geq 2^{\frac{1-r}{r}} \gamma^{\frac{1}{r}} \delta.  
%	\] 
%\end{proof}

\subsection{Locally bounded solutions} \label{sec:bdd}

We assume in this section that $v_-(\omega)>0$ {\it and }$\omega$ is big.
We fix $\rho$ an $\omega$-psh function   such that
\begin{itemize}
\item $\rho$ has analytic singularities and $\sup_X \rho=0$;
\item $\omega+dd^c \rho \geq \delta \omega_X$ is a hermitian current;
%\item $\rho$ is smooth in a Zariski open set $\Omega\subset X$. % $X \setminus E$, where $E=\sum_{j=1}^s a_j E_j$ is a snc effective divisor.
\end{itemize}

Given $\p$ a quasi-plurisubharmonic function on $X$ and $c>0$, we set
$$
E_c(\p):=\{ x \in X  \; : \;  \nu(\p,x) \geq c \},
$$
where $\nu(\p,x)$ denotes the  Lelong number of $\p$ at $x$. 
A celebrated theorem of Siu ensures that for any $c>0$, the
set $E_c(\p)$ is a closed analytic subset of $X$.

\begin{thm}   \label{thm:bdd}
Assume $\omega$ is a big form and $v_-(\omega)>0$.
Fix $p>1$ and $q=\frac{p}{p-1}$.
Let $\mu=fdV_X$ be a probability measure, where
   $f=g e^{-\p}$ with $0 \leq g \in L^p(dV_X)$,  and  $\p$ is 
quasi-plurisubharmonic. Assume $f$ belongs to a fixed Orlicz class $L^w$. 
%function.
Then  there exist a unique constant $c>0$ and a function $\f\in \PSH(X,\omega)$ 
such that
\begin{itemize}
\item $\f$ is locally bounded and continuous in 
%the open set 
$\Omega:= X \setminus \{\rho=-\infty\} \cup E_{q^{-1}}(\p)$;
%where $\frac{1}{p}+\frac{1}{q}=1$;
\item $
  (\omega+dd^c \f)^n=c   fdV_X 
 \; \;  \text{ in } \; \;
\Omega;
%= \{\delta \psi +\rho >-\infty \};
$
\item for any $\alpha>0$, 
$
\alpha ( A^{-1}\p+\rho)-\beta \leq \f \leq 0;
$
\end{itemize}
where $A>0$ is a constant such that $dd^c \p \geq -A \omega_{\rho}$ and 
$\beta>0$ depends on  $p,w$,  upper bounds for $\alpha^{-1}$, $||g||_{L^p}, ||f||_{w}$, $v_-(\omega)^{-1}$.
\end{thm}

\begin{proof}
{\it Reduction to analytic singularities.}
We let $q$ denote the conjugate exponent of $p$,  set $r=\frac{2p}{p+1}$,
and note that $1 < r< p$.
If the Lelong numbers of $\p$ are all less than $\frac{1}{q}$, it follows from H\"older inequality that
$f \in L^r(dV_X)$, since
$$
\int_X f^r dV_X =\int_X g^r e^{-r \p} dV_X
\leq \left( \int_X g^p dV_X \right)^{\frac{r}{p}} \cdot 
\left( \int_X e^{-\frac{pr}{p-r} \p} dV_X \right)^{\frac{p-r}{p}} ,
$$
  where the last integral is finite by Skoda's integrability theorem  \cite[Theorem 8.11]{GZbook}
 if  $\frac{pr}{p-r} \nu(\p,x)<2$
 for all $x \in X$, which is equivalent to $\nu(\p,x)<\frac{1}{q}$.
 
 It is thus natural to expect that the solution $\f$ will be locally bounded in the complement 
 of the closed analytic set $E_{q^{-1}}(\p)$. 
 It follows from Demailly's equisingular approximation technique (see \cite{Dem15}) that
 there exists a sequence $(\p_m)$ of quasi-psh functions on $X$ such that
 \begin{itemize}
 \item $\p_m \geq \p$ and $\p_m \rightarrow \p$ (pointwise and in $L^1$);
 \item $\p_m$ has analytic singularities, it is smooth in $X \setminus E_{m^{-1}}(\p)$;
 \item $dd^c \p_m \geq -K \omega_X$, for some uniform constant $K>0$;
 \item $\int_X e^{2m(\p_m-\p)} dV_X <+\infty$ for all $m$.
 \end{itemize}
  We choose $m=[q]$, set
  $g_m:=g e^{\p_m-\p}$, and observe that
 \begin{eqnarray*}
 \int_X g_m^r  & \leq  &
 \left( \int_X e^{2m(\p_m-\p)}dV_X \right)^{\frac{1}{2m}} \cdot 
\left( \int_X   g_m^{\frac{2mr}{2m-r}}  dV_X \right)^{\frac{2m-r}{2m}}  \\
& \leq & \left( \int_X e^{2m(\p_m-\p)}dV_X \right)^{\frac{1}{2m}} \cdot 
\left( \int_X   g_m^{p}  dV_X \right)^{\frac{2m-r}{2m}} <+\infty
 \end{eqnarray*}
  if we choose $r^{-1}=p^{-1}+(2m)^{-1}<1$ so that $\frac{2mr}{2m-r}=p$.
 By replacing $\p$ by $\p_{[q]} \geq \p$ and $g$ by $g_m \in L^r$
     in the sequel, we can thus  
  assume that $\p$
  has analytic singularities and
   is smooth in $X \setminus E_{q^{-1}}(\p)$.

\medskip

\noindent  {\it Bounded approximation.}
% We can assume 
 %without loss of generality  that $\p\leq 0$.
 The densities $f_j:= g e^{-\max(\psi,-j)}$ belong to $L^p$ and
  increase towards $f$. By Theorem \ref{thm: bounded sol semipositive}
  there exists $\f_j \in \PSH(X,\omega) \cap L^{\infty}(X)$ such that
  \[
  (\omega+dd^c \f_j)^n = c_j f_j dV_X, 
  \]
where $c_j>0$ and  $\sup_X \f_j=0$.
  
Observe that $dV_X \geq \delta_1 \omega_X^n$ for some $\delta_1>0$. 
 It follows therefore from the arithmetico-geometric means inequality that 
 $$
 c_j^{1/n} f_j^{1/n} \delta_1^{1/n} \omega_X^n \leq  (\omega+dd^c \f_j) \wedge \omega_X^{n-1}.
 $$
 As $j$ increases to $+\infty$, we have  $\int_X f_{j}^{1/n}  \omega_X^n \rightarrow \int_X f^{1/n} \omega_X>0$, while
 $$
 \int_X (\omega+dd^c \f_j) \wedge \omega_X^{n-1} \leq  \int_X \omega \wedge \omega_X^{n-1} 
 +B \int_X (-\f_j)  \omega_X^{n},
 $$
 using that $dd^c \omega_X^{n-1}  \leq B \omega_X^n$ and $\f_j \leq 0$.
 Thus the positive constants $c_j$ are uniformly bounded from above
  (the functions $\f_j$ are   $\omega$-psh and normalized, hence they 
 are relatively compact in $L^1(X)$).
 
 They are also 
 %uniformly 
 bounded   away from zero, since
 $$
 c_{j} \int_X f_j dV_X =\int_X  (\omega+dd^c \f_j)^n \geq v_-(\omega) ,
 $$
 and $\int_X f_j dV_X  \rightarrow \int_X f dV_X =1$.

  \medskip
 
\noindent  {\it $L^{\infty}$-estimates.}
Fix $A>0$ large enough so that $\omega+dd^c \p \geq -A \omega_{\rho}$.
Since $c \leq c_j \leq C $,  $f_j \leq f$ where $f$ belongs to a fixed Orlicz class $L^w$,
$f_j \leq ge^{-A \tilde{\p}}$ where $\tilde{\p}=\frac{\p}{A}+\rho \in \PSH(X,\omega)$,
 it follows from Theorem \ref{thm:c0relative} that 
  \[
  \alpha (A^{-1}\psi+\rho) - \beta \leq \f_j\leq 0. 
  \]
Thus  the functions $\f_j$ are locally uniformly bounded in $\Omega:=X \setminus \{\rho+\p=-\infty\}$.

We can thus extract a subsequence  such that
  $c_{j} \longrightarrow  c>0$ and  
\begin{itemize}
\item $\f_j \rightarrow \f$ a.e. and in $L^1$, with $\f \in \PSH(X,\omega)$; 
\item $\sup_X \f=0$, as follows from Hartogs lemma (see \cite[Theorem 1.46]{GZbook});
\item $ \alpha (A^{-1}\p+\rho)-\beta \leq \f \leq 0$ in $\Omega$,
\end{itemize}
hence the Monge-Amp\`ere measure $(\omega+dd^c \f)^n$ is well-defined in $\Omega$.

  \medskip
 
 \noindent  {\it Convergence of the approximants.} 
  Set
$$
\Phi_j^+:=\left( \sup_{\ell \geq j} \f_{\ell} \right)^*
\; \; \text{ and } \; \; 
\Phi_j^-:=P_{\omega}\left( \inf_{\ell \geq j} \f_{\ell} \right).
$$
The function $\Phi_j^+$ (resp. $\Phi_j^-$) is $\omega_j$-psh (resp. $\omega$-psh) and  locally bounded in $\Omega$.
Note that $\Phi_j^+$ decreases to $\f$, while $\Phi_j^-$ increases to some $\omega$-psh function
$\Phi' \leq \f$. We are going to show that $\Phi'=\f$.

We also set
$
f_j^+:= \sup_{\ell \geq j} f_{\e_{\ell}}, \;
f_j^-:= \inf_{\ell \geq j} f_{\e_{\ell}}, \;
%\; \; \text{ and } \; \;
c_j^+=\sup_{\ell \geq j} c_{\e_{\ell}}, \;
c_j^-=\inf_{\ell \geq j} c_{\e_{\ell}}.
$
 It follows from Lemma \ref{lem:envmin} that
 $$
(\omega+dd^c \Phi_j^+)^n \geq c_j^- f_j^- dV_X.
 $$
 Since $\Phi_j^+$ decreases to $\f$, while $c_j^- f_j^-$ increases to $cf$, it follows from the continuity
 of the complex Monge-Amp\`ere along monotonic sequences that
 $$
(\omega+ dd^c \f)^n \geq c f dV_X
\; \; \text { in } \; \;
\Omega.
 $$

  \smallskip

 It follows from Lemma \ref{lem:envmin} that
 $$
(\omega+dd^c \Phi_j^+)^n \geq c_j^- f_j^- dV_X.
 $$
 Since $\Phi_j^+$ decreases to $\f$, while $c_j^- f_j^-$ increases to $cf$, it follows from the continuity
 of the complex Monge-Amp\`ere along monotonic sequences that
 $$
(\omega+ dd^c \f)^n \geq c f dV_X
\; \; \text { in } \; \;
\Omega.
 $$
It also follows  from Lemma \ref{lem:envmin} that
 $$
(\omega+ dd^c \Phi_j^-)^n \leq e^{\Phi_j^--\inf_{\ell \geq j} \f_j}c_j^+ f_j^+ dV_X,
 $$
 so $\Phi_j^-$ increase to some $\Phi' \in \PSH(X,\omega)$ such that 
 $$
(\omega+ dd^c \Phi')^n \leq  e^{\Phi'-\f}c f dV_X \leq e^{\Phi'-\f} (\omega+dd^c \f)^n
 $$
 in $\Omega$.
 It thus follows from  Proposition \ref{prop: domination principle unbounded}, that $\Phi' \geq \f$, hence $\Phi' =\f$
 since by construction $\Phi' \leq \f$.
Altogether this shows that $\f$  solves the desired equation.

 \smallskip
 
  The   uniqueness of the constant $c>0$    follows from Corollary \ref{cor:unique2}.

% \noindent  {\it Uniqueness of the Monge-Amp\`ere constant $c$.} 
% The constant $c>0$ above is unique. Indeed
% assume there are constants $0<c_1<c_2$   and $u,v \in PSH(X,\omega)$ such that 
 
%% (1) for all $\varepsilon >0$, $\inf_X (\min(u,v)-\varepsilon \rho) >-\infty$; 
 
% (2) $(\omega+dd^c u)^n = c_1 fdV$ and $(\omega+dd^c v)^n = c_2 fdV$. 
 
%\noindent  Proposition \ref{prop: domination principle unbounded} would then imply that $u \geq v-C$ for all $C>0$, a contradiction. 
 
 \medskip
 
 \noindent  {\it Continuity of the solutions.}  
 We finally prove that any  solution $\phi \in \PSH(X,\omega)$ satisfying 
 $\inf_X(\phi - \varepsilon \rho) >-\infty$, for all $\varepsilon>0$, is continuous in $\Omega$. 
We can assume without loss of generality that $\phi \leq -1$. 
%Recall that we assume $\psi$ has analytic singularities   and $\psi$ is smooth in $\Omega= \{\psi>-\infty\}$.
  Note that $\phi$ has zero Lelong number everywhere on $X$.  
   Using Demailly's regularization we can thus find $\phi_j$  smooth functions  
   such that $\omega+dd^c \phi_j \geq -2^{-j} \omega_X$, and $\phi_j \searrow \phi$. Fix $\varepsilon >0$ and solve 
 \[
 (\omega+dd^c u_{j,\varepsilon})^n = e^{u_{j,\varepsilon}}(g_{j,\varepsilon} e^{-2\varepsilon^{-1} \phi} +h) dV_X,
 \]
 where $h= \omega^n/dV_X$, $g_{j,\varepsilon} = \varepsilon^{-n}{\bf 1}_{\{\phi <\phi_j - \varepsilon\}}g$. 
 Observe that  $g_{j,\varepsilon} e^{-\varepsilon^{-1} \phi} \to 0$   in $L^p(X)$ as $j\to +\infty$.
 It follows from Theorem \ref{thm: stab sem pos} that
   the functions $u_{j,\varepsilon}$  uniformly converge to $0$ as $j\to +\infty$. Consider 
 \[
 v_{j,\varepsilon} := (1-\varepsilon) \phi_j  + \frac{\varepsilon}{2} (\psi + A\rho + u_{j,\varepsilon}) - C\varepsilon,
 \]
where $A>0$ is a constant ensuring that $\omega+dd^c (\psi + A\rho) \geq 0$, and $C>0$ is a constant chosen below. 
 A direct computation shows that 
 \[
 (\omega+dd^c v_{j,\varepsilon})^n \geq 2^{-n} e^{u_{j,\varepsilon}} {\bf 1}_{\{\phi <\phi_j - \varepsilon\}}g e^{-\varepsilon^{-1} \phi}dV_X.  
 \]
 Recall that, by Theorem \ref{thm:c0relative}, $\phi_j \geq \phi \geq \frac{1}{2} (\psi +A\rho) -C_1$, for some positive constant $C_1$. 
 On the set $\{\phi < v_{j,\varepsilon}\}$ we have 
 \[
 \phi < \phi_j - \varepsilon \left(\phi_j - \frac{1}{2} (\psi +A\rho)\right) + \frac{\varepsilon}{2}\sup_X |u_{j,\varepsilon}| -C\varepsilon<\phi_j- \varepsilon, 
 \]
 if we choose $C>1+ \sup_X |u_{j,\varepsilon}| +C_1$. On this set we also have 
 \[
 2^{-n} e^{u_{j,\varepsilon}} g e^{-\varepsilon^{-1} \phi} \geq 2^{-n} e^{-\sup_X |u_{j,\varepsilon}|} e^{C - \psi} \geq 2 g e^{-\psi}, 
 \]
 if we choose $C>2^n + \sup_X|u_{j,\varepsilon}|+ \ln(2)$. Thus
 \[
 (\omega+dd^c \phi)^n \leq 2^{-1}(\omega+dd^c v_{j,\varepsilon})^n = 2^{-1}(\omega+dd^c \max(\phi,v_{j,\varepsilon}))^n  
 \]
 on the set $\{\phi<v_{j,\varepsilon}\}$. 
Thus  $\phi \geq \max(\phi,v_{j,\varepsilon})\geq v_{j,\varepsilon}$
by  Proposition \ref{prop: domination principle unbounded}. 
Fixing a compact set $K\Subset \Omega$ and letting $j\to +\infty$ we obtain 
 \[
 \liminf_{j\to +\infty}\inf_K (\phi- \phi_j) \geq -C'\varepsilon. 
 \]
 Since $\phi_j\to \phi$ uniformly  on $K$ as $\varepsilon\to 0$, we conclude that $\phi \in {\mathcal C}^0(\Omega)$. 
\end{proof}

\begin{remark}
This relative $L^{\infty}$-bound requires one to assume both that $\omega$ is   big and $v_-(\omega)>0$.
It is natural to expect that it suffices to assume that $\omega$ is big, however such an estimate is an open
problem even in the simplest case when $\omega$ is a hermitian form !
On the other hand these two conditions are expected to be equivalent; such is the case e.g.
when $X$ belongs to the Fujiki class (see \cite{GL21b}).
\end{remark}

  \section{Geometric applications} 
 
In this section we apply the previous analysis on compact complex varieties $V$ with {\it mild singularities}.
 We refer the reader to \cite[Section 5]{EGZ09} for an introduction to complex analysis
 on singular varieties. Roughly speaking one can consider
 local embeddings $j_{\alpha} V_{\alpha} \hookrightarrow \C^N$
 and consider objects (quasi-psh functions, forms, etc) that are restrictions of global ones.
 
One can also consider $\pi:X \rightarrow V$ a resolution of singularities and pull-back
these objects to $X$ in order to study the corresponding equations in a smooth environment.
The draw-back is a loss of positivity along some divisor $E=\pi^{-1}(V_{sing})$ which lies above
the singular locus of $V$. Considering a hermitian form $\omega_V$ on $V$, we are thus lead
to work with $\omega=\pi^* \omega_V$, which is  semi-positive and big
but no longer hermitian. We fix a $\omega$-psh function $\rho$ with analytic singularities 
along a divisor $E$ such that $\omega+dd^c \rho \geq 3 \delta \omega_X$, for some $\delta>0$.

 \subsection{Higher order regularity of  solutions} \label{sec:reg}
 
 In this section we establish  higher order regularity of the solution
 under appropriate assumptions on the density.

   \begin{thm} \label{thm:higher0}
  Assume $\omega$ is big or $v_-(\omega)>0$.
  Assume $ f=e^{\p^+-\p^-} \in L^p(dV_X)$,  $p>1$, and $f$ is smooth and positive in $X \setminus D$, with
$\p^{\pm}$     quasi-plurisubharmonic functions.
%    with $\p^- \in PSH(X, a \omega)$ for some $a>0$.
Then there exist   $c>0$ and $\f \in \PSH(X,\omega)$ such that $\sup_X \f=0$ and
\begin{itemize}
\item $\f$ is smooth in the open set $X \setminus (D \cup E)$;
\item $-T \leq \f \leq 0$ in $X$ for some uniform $T>0$;
\item $(\omega+dd^c \f)^n=c fdV_X$ on $X$.
\end{itemize} 
 \end{thm}

 We also establish higher regularity under less restrictive assumptions on the density
 $f$, but we then need  to assume that $\omega$ is both big and satisfies $v_-(\omega)>0$,
 in order to have a good relative $L^{\infty}$-bound:

  \begin{thm} \label{thm:higher}
  Assume $\omega$ is big and $v_-(\omega)>0$.
  Assume $ f=e^{\p^+-\p^-} \in L^1(dV_X)$ is smooth and positive in $X \setminus D$, with
$\p^{\pm}$     quasi-plurisubharmonic functions.
%    with $\p^- \in PSH(X, a \omega)$ for some $a>0$.
Then there exist   $c>0$ and $\f \in \PSH(X,\omega)$ such that $\sup_X \f=0$ and
\begin{itemize}
\item $\f$ is smooth in the open set $X \setminus (D \cup E)$;
\item for any $\alpha>0$ there is $\beta(\alpha)>0$ with $\alpha(\delta \p^-+\rho) -\beta(\alpha) \leq \f \leq 0$ in $X$;
\item $(\omega+dd^c \f)^n=c fdV_X$ in $X \setminus (D \cup E)$.
\end{itemize} 
 \end{thm}
 
One can keep track of the dependence of the constants   on the data,
as was done in Theorem \ref{thm:bdd}.
  This   result can be seen as a generalization of the main result of \cite{TW10} which dealt with the case
 when $\omega>0$ is hermitian on $X$ and $f$ is globally smooth and positive.

 \begin{proof} 
 We prove Theorem \ref{thm:higher0} and Theorem \ref{thm:higher} at once.
 
 \smallskip
 
\noindent  {\it Smooth approximation.}
 We assume without loss of generality that $\p^{\pm} \leq 0$. For a Borel function $g$ we let $g_{\e}:=\rho_{\e}(g)$ denote the Demailly regularization of $g$ (see \cite[(3.1)]{Dem94}):
 
\[
\rho_{\e}(g)(x) :=  \frac{1}{\e^{2n}} \int_{\zeta \in T_xX} g({\rm exph}_x(\zeta)) \chi\left (\frac{|\zeta|^2}{\e^2} \right ) d\lambda (\zeta). 
\]
  Since $\p^-$ is quasi-psh, the corresponding 
regularization $\p^-_{\e}$ satisfies $\p^- \leq \p^-_{\e}+A \e^2$,
 while $\p^+_{\e} \leq0$. 
 Moreover the functions $\p_\e^{\pm}$ are quasi-psh with
$dd^c \p_\e^{\pm} \geq -K^{\pm} \omega_X$ for uniform constants $K^{\pm} \geq 0$. In particular
$$
K^- \omega+ dd^c \left( \delta \p_\e^{-}+K^- \rho \right) \geq 
-K^-\delta \omega_X+K^- (\omega+dd^c \rho) \geq 0,
$$
so $\alpha \p_\e^{-}+ \rho \in \PSH(X, \omega)$ for all $0<\alpha< \frac{\delta}{K^-}$. Up to replacing $\delta$ with $\frac{\delta}{K^-}$,  we assume in the sequel that $K^-=1$. 
 
\smallskip

We fix $0< \e \leq 1$ and set $\omega_{\e}:=\omega+\e \omega_X$.
 It follows from   \cite{TW10} that there exist unique constants $c_\e>0$ and  smooth
 $\omega_{\e}$-psh functions $\f_\e$   such that
 $$
 (\omega+\e \omega_X+dd^c \f_\e)^n=c_\e e^{\psi^+_{\e}-\psi^-_{\e}} dV_X.
 $$
 with  $\sup_X \f_\e=0$.
 Note that by Jensen inequality we have 
$$
e^{\psi^+_{\e}-\psi^-_{\e}}\leq  f_{\e}.
$$
Since $f_{\e}$ converges to $f$ in $L^1(X,dV_X)$, extracting a subsequence we can assume that  $e^{\psi^+_{\e}-\psi^-_{\e}}\leq F$, for some $F\in L^1(dV_X)$ which  belongs to an Orlicz class $L^{w}$, for some convex increasing weight $w:\R^+ \rightarrow \R^+$ such that $\frac{w(t)}{t} \rightarrow +\infty$ as $t \rightarrow +\infty$. 

\medskip

\noindent  {\it ${\mathcal C}^0$-estimates.} 
This is the only place where the proofs of Theorem \ref{thm:higher0} and Theorem \ref{thm:higher} slightly differ,
and also the reason why the assumptions on $\omega$ look slightly stronger in the second case.

Arguing as in the proof of Theorem \ref{thm:bdd}, one shows that the   constants $c_{\e}$ are uniformly bounded   away from zero and $+\infty$. 
If $f$ belongs to $L^p(dV_X)$ for some $p>1$, it follows from Theorem \ref{thm:uniformhermitian} (if $v_-(\omega)>0$)
and Theorem \ref{thm: bounded sol semipositive} (if $\omega$ is big) that  $-T \leq \f_{\e} \leq 0$, for some uniform constant $T$
independent of $\e>0$.

If we rather assume that $f \in L^1(dV_X)$ with $\omega$ big {\it and}  $v_-(\omega)>0$, 
note that 
$$
(\omega+\varepsilon \omega_X+ dd^c \f_{\varepsilon})^n \leq Ce^{-A(\delta \psi^-_{\e}+ \rho)}dV_X, 
$$
with $A=\delta^{-1}$.
 It then follows from Theorem \ref{thm:c0relative} that
for any $\alpha>0$ there exists $\beta(\alpha)$ such that
 \begin{equation}
	\label{eq: C0 estimate}
  \alpha (\delta\p^-_{\e} +\rho)-\beta(\alpha) \leq \f_{\e} \leq 0.
\end{equation}
Thus $(\f_{\e})$ is locally uniformly bounded in $\Omega:=X \setminus (E \cup D)$.

  \medskip
 
\noindent  {\it ${\mathcal C}^2$-estimates.} 
In the sequel we fix $\alpha=1$
and we establish uniform bounds on $\Delta_{\omega_X} \f_\e$ on compact subsets of $\Omega$.
% Since $\omega+dd^c \rho$ is a hermitian form in $\Omega$, the forms 
% $\beta_{\e}=\omega+dd^c \rho+\e \omega_X$
% and $\omega_X$ are uniformly comparable on compact subsets of $\Omega$, so it suffices
% to bound from above ${\rm Tr}_{\beta_{\e}}(\omega_{\e}+dd^c \f_{\e})$.
 We follow the computations of \cite[Proof of Theorem 2.1]{TW10a} and \cite{To18} 
 with a twist in order to deal with unbounded functions. 
 We use $C$ to denote various uniform constants which may be different. Consider
 $$
 H:=\log {\rm Tr}_{\omega_X}(\tilde{\omega})- \gamma(u)
 $$
 where 
 $$
 \tilde{\omega}=\omega_{\e}+dd^c \f_{\e},
 \; \;
 u= (\f_{\e}-\rho-2\delta\psi^-_\e+\beta +1)>1, 
 $$
 and $\gamma: \mathbb{R} \rightarrow \mathbb{R}$ is a smooth concave increasing function such that $\gamma(+\infty)=+\infty$.
 % which will be chosen later. 
 We are going to show that $H$ is uniformly bounded from above for an appropriate choice of
 $\gamma$. Since $u$ is uniformly bounded on compact subsets of $\Omega$, this will
 yield for each $\e>0$ and $K \subset \subset \Omega$ a uniform bound
 $$
 \Delta_{\omega_X}(\f_{\e})=
 {\rm Tr}_{\omega_X}(\omega_{\e}+dd^c \f_{\e}) - {\rm Tr}_{\omega_X}(\omega_{\e}) \leq C_K.
 $$
 
 We let ${g}$ denote the Riemannian metric associated to $\omega_X$
 and $\tilde{g}$ the one associated to $\tilde{\omega}=\omega_\e+dd^c \f_{\e}$. 
 To simplify notations we will omit the subscript $\varepsilon$ in the sequel. Since $\psi^{\pm}$ are quasi-psh, up to multiplying $\omega_X$ with a large constant, we can assume that $\omega_X+dd^c \psi^{\pm} \geq 0$. 
 
 Since $\rho \rightarrow -\infty$ on 
 %$\partial \Omega$, 
 $E$
 the maximum of $H$ is attained at some point $x_0 \in X \setminus E$.  
 We use special 
 %local 
 coordinates at this point, as defined by Guan-Li in \cite{GL10}: 
 \[
 g_{i\bar{j}} = \delta_{ij}, \; \;  
 \frac{\partial g_{i\bar{i}}}{\partial z_j}=0 
 \; \; \text{and}\; \; \tilde{g}_{i\bar{j}}\; \text{is diagonal}.
 \]
 To achieve this we use a linear change of coordinates so that $g_{i\bar{j}}=\delta_{ij}$ and $\tilde{g}_{i\bar{j}}$ is diagonal at $x_0$,
 and we  then make a change of coordinates as in \cite[(2.19)]{GL10}. 
 %In $\Omega$, the complex Monge-Amp\`ere equation then reads 
% \begin{equation}
% 	\label{eq: MA coordinate}
% 	 \log \det (\tilde{g}_{i\bar{j}}) = \log c + \psi^+ - \psi^- +\log \det (g_{i\bar{j}}).
% \end{equation}
 We first compute 
 \begin{flalign}
 	\Delta_{\tilde{\omega}} {\rm Tr}_{\omega_X}(\tilde{\omega})
 	&= \sum_{i,j,k,l} \tilde{g}^{i\bar{j}}\partial_i \partial_{\bar{j}} (g^{k\bar{l}} \tilde{g}_{k\bar{l}})\nonumber \\
 	&=\sum_{i,k}\tilde{g}^{i\bar{i}} \tilde{g}_{k\bar{k}i\bar{i}} -2 \Re\left (\sum_{i,j,k} \tilde{g}^{i\bar{i}}g_{j\bar{k}\bar{i}} \tilde{g}_{k\bar{j}i}\right) + \sum_{i,j,k} \tilde{g}^{i\bar{i}}g_{j\bar{k}i} g_{k\bar{j}\bar{i}} \tilde{g}_{k\bar{k}}\nonumber \\
 	&+\sum_{i,j,k} \tilde{g}^{i\bar{i}}g_{k\bar{j}i} g_{j\bar{k}\bar{i}} \tilde{g}_{k\bar{k}} -\sum_{i,k} \tilde{g}^{i\bar{i}}g_{k\bar{k}\bar{i}i}\tilde{g}_{k\bar{k}}\nonumber\\
 	&\geq \sum_{i,k}\tilde{g}^{i\bar{i}} \tilde{g}_{k\bar{k}i\bar{i}} -2 \Re\left (\sum_{i,j,k} \tilde{g}^{i\bar{i}}g_{j\bar{k}\bar{i}} \tilde{g}_{k\bar{j}i}\right) -C{\rm Tr}_{\omega_X}(\tilde{\omega}){\rm Tr}_{\tilde{\omega}}(\omega_X). \nonumber %\label{eq: C2 est 1} 
 \end{flalign}
Using this and 
\[
{\rm Tr}_{\omega_X} {\rm Ric}(\tilde{\omega})=\sum_{i,k} \tilde{g}^{i\bar{i}} \left (-\tilde{g}_{i\bar{i}k\bar{k}} + \sum_j \tilde{g}^{j\bar{j}}\tilde{g}_{i\bar{j}k}\tilde{g}_{j\bar{i}\bar{k}}\right)
\]
we obtain
\begin{flalign}
		\Delta_{\tilde{\omega}} {\rm Tr}_{\omega_X}(\tilde{\omega})& \geq \sum_{i,j,k} \tilde{g}^{i\bar{i}} \tilde{g}^{j\bar{j}}\tilde{g}_{i\bar{j}k}\tilde{g}_{j\bar{i}\bar{k}}-{\rm Tr}_{\omega_X} {\rm Ric}(\tilde{\omega}) \label{eq: C2 est 1} - C \Tr_{\tilde{\omega}}(\omega_X) \\
		&-C{\rm Tr}_{\omega_X}(\tilde{\omega}){\rm Tr}_{\tilde{\omega}}(\omega_X)  
		 -2 \Re\left (\sum_{i,j,k} \tilde{g}^{i\bar{i}}g_{j\bar{k}\bar{i}} \tilde{g}_{k\bar{j}i}\right),\nonumber
\end{flalign}
noting that $|\tilde{g}_{i\bar{i}k\bar{k}}-\tilde{g}_{k\bar{k}i\bar{i}}|\leq C$. Here ${\rm Ric}(\tilde{\omega})$ is the Chern-Ricci form of $\tilde{\omega}$. 
From ${\rm Tr}_{\omega_X}(\tilde{\omega}){\rm Tr}_{\tilde{\omega}}(\omega_X)\geq n$,
  \eqref{eq: C2 est 1} and 
$$
{\rm Ric}(\tilde{\omega})= {\rm Ric}(\omega_X) - dd^c (\psi^+-\psi^-) \leq C \omega_X + dd^c \psi^-,
$$
we obtain 
\begin{flalign}
		\Delta_{\tilde{\omega}} {\rm Tr}_{\omega_X}(\tilde{\omega})& \geq \sum_{i,j,k} \tilde{g}^{i\bar{i}} \tilde{g}^{j\bar{j}}\tilde{g}_{i\bar{j}k}\tilde{g}_{j\bar{i}\bar{k}}-{\rm Tr}_{\omega_X} (\omega_X+dd^c \psi^-) - C\Tr_{\tilde{\omega}}(\omega_X) \label{eq: C2 est 2}\\
		&  -2 \Re\left (\sum_{i,j,k} \tilde{g}^{i\bar{i}}g_{j\bar{k}\bar{i}} \tilde{g}_{k\bar{j}i}\right) - C{\rm Tr}_{\omega_X}(\tilde{\omega}){\rm Tr}_{\tilde{\omega}}(\omega_X). \nonumber
\end{flalign}
Our special choice of coordinates at $x_0$ ensures that $g_{j\bar{j}\bar{i}}=0$. 
Using Cauchy-Schwarz inequality and $|\tilde{g}_{k\bar{j}i}-\tilde{g}_{i\bar{j}k}|\leq C$, we therefore obtain
\begin{flalign*}
	\left | 2 \Re\left ( \sum_{i,j,k} \tilde{g}^{i\bar{i}}g_{j\bar{k}\bar{i}} \tilde{g}_{k\bar{j}i}\right)\right| 
	&\leq  \left | 2 \Re\left (\sum_{i}\sum_{j\neq k} \tilde{g}^{i\bar{i}}g_{j\bar{k}\bar{i}} \tilde{g}_{i\bar{j}k}\right) \right | 
	+ C \Tr_{\tilde{\omega}}(\omega_X)\\
	&\leq \sum_{i}\sum_{j\neq k}  \left ( \tilde{g}^{i\bar{i}}\tilde{g}^{j\bar{j}} \tilde{g}_{i\bar{j} k} \tilde{g}_{j\bar{i}\bar{k}}+  \tilde{g}^{i\bar{i}}\tilde{g}_{j\bar{j}} g_{j\bar{k}\bar{i}} g_{k\bar{j}i}\right) + C \Tr_{\tilde{\omega}}(\omega_X)\\
	& \leq \sum_{i}\sum_{j\neq k} \tilde{g}^{i\bar{i}}\tilde{g}^{j\bar{j}} \tilde{g}_{i\bar{j}k} \tilde{g}_{j\bar{i}\bar{k}} + C {\rm Tr}_{\omega_X}(\tilde{\omega}){\rm Tr}_{\tilde{\omega}}(\omega_X)+ C \Tr_{\tilde{\omega}}(\omega_X).
\end{flalign*} 
Together with \eqref{eq: C2 est 2} this yields
% \begin{flalign}
\begin{equation}\label{eq: C2 estimate 3}
 \Delta_{\tilde{\omega}}
 {\rm Tr}_{\omega_X} (\tilde{\omega}) 
\geq   I -C {\rm Tr}_{\omega_X} (\omega_X+dd^c \psi^-) -C {\rm Tr}_{\omega_X} (\tilde{\omega})  {\rm Tr}_{\tilde{\omega}}  (\omega_X) -C  \Tr_{\tilde{\omega}}(\omega_X) 
\end{equation}
% \end{flalign}
 with $I:=\sum_{i,j} \tilde{g}^{i \bar{i}}  \tilde{g}^{j \bar{j}} \tilde{g}_{i \bar{j}j}   \tilde{g}_{j \bar{i} \bar{j}}$. We next compute 
\begin{flalign*}
	|\partial \Tr_{\omega_X}(\tilde{\omega}) |_{\tilde{\omega}}^2&=\sum_{i,j,k} \tilde{g}^{i\bar{i}} \tilde{g}_{j\bar{j}i} \tilde{g}_{k\bar{k}\bar{i}}\\
	&= \sum_{i,j,k} \tilde{g}^{i\bar{i}}(T_{ij\bar{j}}+ \tilde{g}_{i\bar{j}j}) (\overline{T}_{ik\bar{k}}+\tilde{g}_{k\bar{i}\bar{k}})\\
	& =\sum_{i,j,k} \tilde{g}^{i\bar{i}} \tilde{g}_{i\bar{j}j} \tilde{g}_{k\bar{i}\bar{k}}  + \sum_{i,j,k} \tilde{g}^{i\bar{i}} T_{ij\bar{j}} \overline{T}_{ik\bar{k}} + 2\Re \left ( \tilde{g}^{i\bar{i}} T_{ij\bar{j}} \tilde{g}_{k\bar{i} \bar{k}} \right).
\end{flalign*}
 where $T_{ij\bar{j}}= \tilde{g}_{j\bar{j}i}-\tilde{g}_{i\bar{j}j}$ is the torsion term corresponding to $\omega+\varepsilon \omega_X$ which is under control: $|T_{ij\bar{j}}|\leq C$. We bound the first term by Cauchy-Schwarz inequality
 \begin{flalign*}
 	\sum_{i,j,k} \tilde{g}^{i\bar{i}} \tilde{g}_{i\bar{j}j} \tilde{g}_{k\bar{i}\bar{k}} &= \sum_{i} \tilde{g}^{i\bar{i}} \left |\sum_{j} \tilde{g}_{i\bar{j}j}\right|^2 \\
 	&\leq \left (\sum_{i} \tilde{g}^{i\bar{i}} \tilde{g}^{j\bar{j}} \tilde{g}_{i\bar{j}j} \tilde{g}_{j\bar{i}\bar{j}} \right) \left (\sum_{j} \tilde{g}_{j\bar{j}}  \right) =I\Tr_{\omega_X}(\tilde{\omega}). 
 \end{flalign*}
 We thus get
  \begin{equation}\label{eq: C2 estimate 4}
\frac{| \partial {\rm Tr}_{\omega_X} (\tilde{\omega}) |^2_{\tilde{\omega}}}{({\rm Tr}_{\omega_X} (\tilde{\omega}))^2}
\leq \frac{I}{{\rm Tr}_{\omega_X} (\tilde{\omega})}
+C\frac{ {\rm Tr}_{\tilde{\omega}}(\omega_X)}{({\rm Tr}_{\omega_X} (\tilde{\omega}))^2}
+\frac{2}{({\rm Tr}_{\omega_X} (\tilde{\omega}))^2}
\Re\left(\sum \tilde{g}^{i \overline{i}}  T_{ij\overline{j}}  \tilde{g}_{k\bar{i}\bar{k}}  \right),
 \end{equation}
 Since $\partial_{\overline{i}} H=0$ at the point $x_0$, we obtain by differentiating $H$ once
  $$
  \tilde{g}_{k\overline{k}\overline{i}}=   {\rm Tr}_{\omega_X} (\tilde{\omega})\gamma'(u)u_{\overline{i}}. 
 $$
Cauchy-Schwarz inequality yields
  \begin{equation*}
\left | \frac{2}{({\rm Tr}_{\omega_X} (\tilde{\omega}))^2}
\Re\left(\sum \tilde{g}^{i \overline{i}}  T_{ij\overline{j}}  \tilde{g}_{k \overline{k} \overline{i}}  \right) \right |
\leq C \frac{\gamma'(u)^2}{(-\gamma''(u))}
\frac{{\rm Tr}_{\tilde{\omega}}(\omega_X)}{({\rm Tr}_{\omega_X} (\tilde{\omega}))^2}
+ (-\gamma''(u)) |\partial u|_{\tilde{\omega}}^2.
 \end{equation*}
Noting that $|\tilde{g}_{k\bar{k}\bar{i}}-\tilde{g}_{k\bar{i}\bar{k}}|\leq C$ we infer
%   \begin{flalign*}
$$
%&
\left | \frac{2}{({\rm Tr}_{\omega_X} (\tilde{\omega}))^2}
\Re\left(\sum \tilde{g}^{i \overline{i}}  T_{ij\overline{j}}  \tilde{g}_{k \overline{i} \overline{k}}  \right) \right |
%\\
%&
\leq C \left (\frac{\gamma'(u)^2}{(-\gamma''(u))} +1\right)
\frac{{\rm Tr}_{\tilde{\omega}}(\omega_X)}{({\rm Tr}_{\omega_X} (\tilde{\omega}))^2}
+ (-\gamma''(u)) |\partial u|_{\tilde{\omega}}^2.
$$
% \end{flalign*}
Since $0 \geq \Delta_{\tilde{\omega}} H$  at $x_0$, it follows from 
 \eqref{eq: C2 estimate 3}, \eqref{eq: C2 estimate 4} that
  \begin{flalign}\label{eq: C2 est final 1}
0 &\geq \Delta_{\tilde{\omega}} H
= 
 \frac{\Delta_{\tilde{\omega}} {\rm Tr}_{\omega_X} (\tilde{\omega})}{ {\rm Tr}_{\omega_X} (\tilde{\omega})}
 -\frac{| \partial {\rm Tr}_{\omega_X} (\tilde{\omega}) |^2_{\tilde{\omega}}}{({\rm Tr}_{\omega_X} (\tilde{\omega}))^2}
 -\gamma'(u) \Delta_{\tilde{\omega}}(u) - \gamma''(u) |\partial u|_{\tilde{\omega}}^2  \\
 &\geq  -\frac{C {\rm Tr}_{\omega_X} (\omega_X +dd^c \psi^-)}{{\rm Tr}_{\omega_X}(\tilde{\omega})} 
- \gamma'(u) (n - \delta{\rm Tr}_{\tilde{\omega}}( 3\omega_X+ 2dd^c \psi^- )) \nonumber \\
&-C \left (\frac{\gamma'(u)^2}{(-\gamma''(u))}+1\right)
\frac{{\rm Tr}_{\tilde{\omega}}(\omega_X)}{({\rm Tr}_{\omega_X} (\tilde{\omega}))^2}-C \frac{\Tr_{\tilde{\omega}}(\omega_X)}{\Tr_{\omega_X}(\tilde{\omega})} -C \Tr_{\tilde{\omega}}(\omega_X) .\nonumber
 \end{flalign}
 
 \medskip
 
  We now choose the function $\gamma$ so as to obtain a simplified information. We set
 $$
 \gamma(u):= \frac{C+1}{\min(\delta,1)} u + \ln(u). 
 $$
 Since $u\geq 1$, we observe that
 $$
 \frac{C+1}{\min(\delta,1)} \leq \gamma'(u) \leq 1+ \frac{C+1}{\min(\delta,1)} 
\; \; \; \text{ and }\;  \; \;
  \frac{\gamma'(u)^2}{|\gamma''(u)|} +1 \leq C_1u^2.
 $$
By incorporating  this into \eqref{eq: C2 est final 1}  we obtain
\begin{flalign*}
	0 &\geq -\frac{C {\rm Tr}_{\omega_X} (\omega_X +dd^c \psi^-)}{{\rm Tr}_{\omega_X}(\tilde{\omega})} -C_2+ (C+1)(\Tr_{\tilde{\omega}}(\omega_X) + \Tr_{\tilde{\omega}}(\omega_X+dd^c \psi^-))  \\
	&-C_2 (u^2+1)
\frac{{\rm Tr}_{\tilde{\omega}}(\omega_X)}{({\rm Tr}_{\omega_X} (\tilde{\omega}))^2}-C \frac{\Tr_{\tilde{\omega}}(\omega_X)}{\Tr_{\omega_X}(\tilde{\omega})} -C \Tr_{\tilde{\omega}}(\omega_X).
\end{flalign*}
Using
  ${\rm Tr}_{\omega_X} (\omega_X +dd^c \psi^-) \leq {\rm Tr}_{\tilde{\omega}} (\omega_X +dd^c \psi^-){\rm Tr}_{\omega_X} (\tilde{\omega})$ we thus arrive at 
 \begin{equation}
 	\label{eq: C2 est final 2}
 	 0 \geq {\rm Tr}_{\tilde{\omega}}(\omega_X) - C_2(u^2+1) \frac{{\rm Tr}_{\tilde{\omega}}(\omega_X)}{({\rm Tr}_{\omega_X} (\tilde{\omega}))^2}-C\frac{\Tr_{\tilde{\omega}}(\omega_X)}{\Tr_{\omega_X}(\tilde{\omega})}- C_2. 
 \end{equation}
 
\noindent  At the point $x_0$ we have the following alternative:
\begin{itemize}
\item  if ${\rm Tr}_{\omega_X}(\tilde{\omega})^2 \geq 4C_2 (u^2+1)+(4C)^2$ then 
$$
C_2(u^2+1)\frac{{\rm Tr}_{\tilde{\omega}}(\omega_X)}{({\rm Tr}_{\omega_X} (\tilde{\omega}))^2} \leq \frac{{\rm Tr}_{\tilde{\omega}}(\omega_X)}{4}
\; \; \text{and} \; \;
 C\frac{\Tr_{\tilde{\omega}}(\omega_X)}{\Tr_{\omega_X}(\tilde{\omega})} \leq \frac{{\rm Tr}_{\tilde{\omega}}(\omega_X)}{4},
$$
hence from \eqref{eq: C2 est final 2} we get  ${\rm Tr}_{\tilde{\omega}}(\omega_X) \leq 2C_2$. 
Now 
\begin{equation*}%\label{eq: compare two trace}
\Tr_{\omega_X}(\tilde{\omega}) 
\leq n \frac{\tilde{\omega}^n}{\omega_X^n} \left( \Tr_{\tilde{\omega}}(\omega_X) \right)^{n-1}
\leq  n(2C_2)^{n-1} c e^{\psi^+-\psi^-}
\end{equation*}
yields ${\rm Tr}_{\omega_X}(\tilde{\omega}) \leq Ce^{\psi^+-\psi^-}$.
It follows therefore from \eqref{eq: C0 estimate} that
\begin{flalign*}
H(x_0) &\leq \log (2C_2) - \psi^- -\frac{C+1}{\min(\delta,1)} (\f - \rho - 2\delta \psi^-)\\
&\leq  	\log (2C_2) - (C+1) (\f -\rho -\delta \psi^-) 
\leq C_3. 
\end{flalign*}
%where in the last line we have used 
%by using \eqref{eq: C0 estimate}. 
\item   If ${\rm Tr}_{\omega_X}(\tilde{\omega})^2 \leq 4C_2 (u^2+1)+(4C)^2$ then
$$
 H(x_0) \leq \log \sqrt{4C_2(u^2+1) +(4C)^2}- \gamma(u) \leq C_4.
 $$
 \end{itemize}
Thus $H(x_0)$ is uniformly bounded from above, yielding the desired estimate.

\medskip

\noindent  {\it Higher order estimates.}
With uniform bounds on $||\Delta_{\omega_X} \f_\e||_{L^{\infty}(K)}$ in hands, we
can use a complex version of Evans-Krylov-Trudinger 
%${\mathcal C}^{\alpha}$
 estimate
(see \cite[Section 4]{TW10a} in this context) and eventually
differentiate the equation to obtain 
-using Schauder estimates-
uniform bounds,
 for each $K \subset \subset \Omega$, $0 < \beta<1$, $j \geq 0$,
$$
\sup_{\e>0} ||\f_{\e}||_{{\mathcal C}^{j,\beta}(K)}=C_{j,\beta}(K) <+\infty,
$$
 which guarantee that $\f_{\e}$ is relatively compact in ${\mathcal C}^{\infty}(\Omega)$.
 
 We now extract a subsequence $\e_j \rightarrow 0$ such that
 \begin{itemize}
 \item $c_{\e_j} \longrightarrow  c>0$;
 \item $\f_{\e_j} \rightarrow \f$ in $L^1$ with $\f \in \PSH(X,\omega)$ and $\sup_X \f=0$ (Hartogs lemma);
 \item $\f \in {\mathcal C}^{\infty}(\Omega)$ with $(\omega+dd^c \f)^n=c fdV_X$ in $\Omega$;
 \item $\f \geq \alpha(\delta \p^-+\rho) -\beta$ in $X$,
 \end{itemize}
 where $\alpha>0$ is arbitrarily small, as follows from \eqref{eq: C0 estimate}.
When  $f \in L^p(dV_X)$, $p>1$,
the solution $\f$ is  even uniformly bounded on $X$.
 \end{proof}
 
 \begin{remark}
 By comparison with the K\"ahler setting, it is not clear how to make sense of the Monge-Amp\`ere measure
 $(\omega+dd^c \f)^n$ across the singularity divisor $D \cup E$ if $\f$ is unbounded,
 and it is delicate to establish uniqueness of the solution, even if the solution is globally bounded
 (see however Theorem \ref{thm: stab sem pos}).
 \end{remark}

 \subsection{Singular hermitian Calabi conjecture} \label{sec:geom}

Let $V$ be a compact complex variety with {\it log-terminal singularities}, i.e. 
$V$ is  a normal complex space such that the canonical bundle $K_V$
is $\Q$-Cartier and for some (equivalently any) resolution of singularities
$\pi:X \rightarrow V$, we have
$$
K_X=\pi^* K_V+\sum_i a_i E_i,
$$
where the $E_i$'s are exceptional divisors with simple normal crossings,
and the rational coefficients $a_i$ 
(the discrepancies) satisfy $a_i>-1$.

Given $\phi$ a smooth metric of $K_V$ and $\sigma$ a
non vanishing  local holomorphic section
of $K_V$, we consider the "adapted volume form"
$$
\mu_{\phi}:=\left( \frac{i^{rn^2} \sigma \wedge \overline{\sigma}}{|\sigma|^2_{r\phi}} \right)^{\frac{1}{r}}.
$$
This measure is independent of the choice of $\sigma$, and it has finite mass on $V$, since
the singularities are  log-terminal.
Given $\omega_V$ a hermitian form  on $V$, there exists a unique  metric  $\phi=\phi(\omega_V)$ 
of $K_V$ such that 
$$
\omega_V^n=\mu_{\phi}.
$$

\begin{defi}
The Ricci curvature form of $\omega_V$ is ${\rm Ric}(\omega_V):=-dd^c \phi$.
\end{defi}

%The first Chern form of the Chern connection of $\omega_V$ is a closed form 
%cohomologousto $c_1(V)$, 
%which we denote by ${\rm Ric}(\omega_V)$.
Recall that the Bott-Chern space $H_{BC}^{1,1}(V,\R)$
is the space of closed real $(1,1)$-forms modulo the image of $dd^c$ acting on real functions.
The form ${\rm Ric}(\omega_V)$ determines a class $c_1^{BC}(V)$ which 
maps to the usual Chern class $c_1(V)$ under the natural surjection $H_{BC}^{1,1}(V,\R) \rightarrow H^{1,1}(V,\R)$.

By analogy with the Calabi conjecture from  K\"ahler geometry, it is natural to wonder conversely, whether
any representative $\eta \in c_1^{BC}(V)$ can be realised as the Ricci curvature
form of a hermitian metric $\omega_V$.
We provide a positive answer, as a consequence of Theorem \ref{thm: bounded sol semipositive} and Theorem \ref{thm:higher0}:

\begin{theorem} \label{thm:calabi}
Let $V$ be a compact hermitian variety with log terminal singularities
equipped with a hermitian form $\omega_V$.
For every smooth closed real $(1,1)$-form   $\eta$ in 
%the first Bott-Chern class 
$c_1^{BC}(V)$, there exists a 
%unique \marginpar{unique?} 
function $\f \in \PSH(V,\omega_V)$ 
%normalized by $\sup_X \f=0$ 
such that
\begin{itemize}
\item $\f$ is globally bounded on $V$ and smooth in $V_{reg}$;
\item $\omega_V+dd^c \f$ is a hermitian form  and ${\rm Ric}(\omega_V+dd^c \f)=\eta$ in $V_{reg}$.
\end{itemize}
\end{theorem}

In particular if $c_1^{BC}(V)=0$, any hermitian form $\omega_V$ is "$dd^c$-cohomologous" to a 
Ricci flat hermitian current. Understanding the asymptotic behavior of these singular Ricci flat currents
near the singularities of $V$ is, as in the K\"ahler case, an important open problem.

\begin{proof}
It is classical that solving the (singular) Calabi conjecture is equivalent to
solving a complex Monge-Amp\`ere equation.
We let $\pi:X \rightarrow V$ denote a log resolution of singularities
and observe that
$$
\pi^* \mu_{\phi}=f dV,
\; \; \text{ where } \; \; 
f=\prod_{i=1}^k |s_i|^{2a_i}
$$
has poles (corresponding to $a_i<0$) or zeroes (corresponding to $a_i>0$)
 along the exceptional divisors $E_i=(s_i=0)$ and
$dV$ is a smooth volume form on $X$.

We set $\p^+=\sum_{a_i>0} 2 a_i\log|s_i|$, $\p^-=\sum_{a_i<0} 2 |a_i|\log|s_i|$,
and fix $\phi$ a smooth metric of $K_V$ such that $\eta=-dd^c \phi$.
Finding $\omega_V+dd^c \f$ such that 
${\rm Ric}(\omega_V+dd^c \f)=\eta$
is thus equivalent to solving the Monge-Amp\`ere equation
$(\omega_V+dd^c \f)^n=c \mu_{\phi}$.
Passing to the resolution this boils down to solve
$$
(\omega+dd^c \tilde{\f})^n
=c e^{\p^+-\p^-} dV
$$
on $X$, where $\omega=\pi^* \omega_V$ and $\tilde{\f}=\f \circ \pi \in \PSH(X,\omega)$.

Since $\omega$ is semi-positive and big, and since $\p^{\pm}$ are quasi-plurisubharmonic functions
which are smooth in $X^0=\pi^{-1}(V_{reg})$, it follows from
Theorem \ref{thm: bounded sol semipositive} and Theorem \ref{thm:higher0}
that there exists a  solution
$\tilde{\f}$ with all the required properties.
The function $\f=\pi_* \tilde{\f}$ is the potential we were looking for.
\end{proof}

\end{document}